\documentclass[a4paper, reqno, 10pt]{amsart}

\usepackage{color}

\usepackage[T1]{fontenc}
\usepackage{tikz}
\usepackage{tikzpagenodes}
\usepackage{mathrsfs}

\usepackage{amsmath, amsfonts, amssymb, amsthm, graphics, graphicx}
\usepackage{hyperref}
\renewcommand{\\}{\vspace{3mm}}

\providecommand{\aut}{\mathop{\rm Aut \,}\nolimits}
\providecommand{\sym}{\mathop{\rm Sym \,}\nolimits}
\providecommand{\Wr}{\mathop{\rm Wr}\nolimits}
\newcommand{\Fix}{\mathop{\rm Fix}\nolimits}

\newcommand{\N}{\mathbb{N}}

\newcommand{\Path}{\mathcal{P}}

\newcommand{\gpa}{M} 
\newcommand{\gpb}{N} 
\newcommand{\sta}{X} 
\newcommand{\stb}{Y} 
\newcommand{\pta}{x} 
\newcommand{\ptb}{y} 
\newcommand{\cda}{m} 
\newcommand{\cdb}{n} 
\newcommand{\T}{T} 
\newcommand{\col}{\mathcal L} 
\newcommand{\altcol}{\mathcal \col'} 
\newcommand{\VV}{V}
\newcommand{\VVsubset}{\Phi}
\newcommand{\VValt}{\Delta}
\newcommand{\va}{p}
\newcommand{\vb}{q}
\newcommand{\ArcsFrom}{A}

\newcommand{\ColouringInducedPerm}[2]{\theta(#1, #2)}

\newcommand{\Universal}[3]{\mathcal{U}_{#3}(#1, #2)}

\newcommand{\Boxproduct}[3]{#1 \boxtimes_{#3} #2}

\newcommand{\Va}{V_{\sta}}
\newcommand{\Vb}{V_{\stb}}

\theoremstyle{plain}
\newtheorem{theorem}{Theorem}
\newtheorem*{theorem*}{Theorem}
\newtheorem{lemma}[theorem]{Lemma}
\newtheorem{corollary}[theorem]{Corollary}
\newtheorem*{corollary*}{Corollary}

\newtheorem{proposition}[theorem]{Proposition}

\theoremstyle{definition}
\newtheorem{definition}[theorem]{Definition}
\newtheorem*{definition*}{Definition}

\newtheorem{remark}[theorem]{Remark}

\newtheorem*{example*}{Example}
\newtheorem*{remark*}{Remark}

\title{\bf A product for permutation groups and topological groups}

\author{Simon M.~Smith}

\address{School of Mathematics and Physics, University of Lincoln, Lincoln, Lincolnshire, U.K.}

\email{SiSmith@lincoln.ac.uk}

\thanks{This research was partially supported by an Australian Research Council Discovery Early Career Researcher Award (project number DE130101521)}

\begin{document}
\maketitle
\markboth{\textsc{A product for permutation groups}}{\textsc{A product for permutation groups}}

\begin{tikzpicture}[remember picture,overlay]
\node[anchor=north east] at (current page.north east) {\parbox{3in}{\flushright Published in: Duke Math. J. Volume 166, Number 15 (2017), 2965-2999.\\ doi:10.1215/00127094-2017-0022}};
\end{tikzpicture}

\begin{abstract}
We introduce a new product for permutation groups. It takes as input two permutation groups, $\gpa$ and $\gpb$, and produces an infinite group $\Boxproduct{\gpa}{\gpb}{}$ which carries many of the permutational
properties of $\gpa$. Under mild conditions on $\gpa$ and $\gpb$ the group $\Boxproduct{\gpa}{\gpb}{}$ is simple.

As a permutational product, its most significant property is the following: $\Boxproduct{\gpa}{\gpb}{}$ is primitive if and only if $\gpa$ is primitive but not regular, and $\gpb$ is transitive.
Despite this remarkable similarity with the wreath product in product action, $\Boxproduct{\gpa}{\gpb}{}$ and $\gpa \Wr \gpb$ are thoroughly dissimilar.

The product provides a general way to build exotic examples of non-discrete, simple, totally disconnected, locally compact, compactly generated topological groups from discrete groups.

We use this to obtain the first construction of uncountably many pairwise non-isomorphic simple topological groups that are totally disconnected, locally compact, compactly generated and non-discrete. The groups we construct all contain the same compact open subgroup.
The analogous result for discrete groups was proved in 1953 by Ruth Camm.

To build the product, we describe a group $\Universal{\gpa}{\gpb}{}$ that acts on a biregular tree $T$. This group has a natural universal property and is 
a generalisation of
the iconic universal group construction of
Marc Burger and Shahar Mozes
for locally finite regular trees.
\end{abstract}

%
%
\section{Introduction}

We introduce a new product for permutation groups, which we call the {\em box product}, and describe some of its striking properties.
Our new product is non-associative.
It takes two permutation groups $\gpa \leq \sym(\sta)$ and $\gpb \leq \sym(\stb)$ such that $|\sta|, |\stb| > 1$, and yields an infinite permutation group $\Boxproduct{\gpa}{\gpb}{}$.
The group $\Boxproduct{\gpa}{\gpb}{}$ is simple under very mild conditions on $\gpa$ and $\gpb$ (see Theorem~\ref{theorem:simplicity}), and it enjoys a universal property (see Theorem~\ref{theorem:product_as_aut_T}(\ref{item:universal})).

The
group $\Boxproduct{\gpa}{\gpb}{}$ inherits  permutational properties from $\gpa$.
For example, consider the important permutational property of {\em primitivity}. Because primitive actions are minimal, products typically do not preserve primitivity. The iconic exception to this is the wreath product in its product action:

\begin{itemize}
\item
	{\em $\gpa \Wr \gpb$ is primitive in its product action on $\sta^{\stb}$ if and only if $\gpa$ is primitive and not regular on $\sta$, and $\gpb$ is transitive and finite.}
\end{itemize}

\noindent Because of this property, the wreath product is the primary tool for building new primitive groups from other primitive groups. It is fundamental to the classification of the finite primitive permutation groups (the O'Nan--Scott Theorem). Compare the above with the following
remarkable
result for the box product:

\begin{itemize}
\item
	{\em $\Boxproduct{\gpa}{\gpb}{}$ is primitive in its natural action if and only if $\gpa$ is primitive and not regular on $\sta$, and $\gpb$ is transitive} (see Theorem~\ref{thm:PermutationalProperties}).
\end{itemize}

\noindent The groups $\gpa \Wr \gpb$ and $\Boxproduct{\gpa}{\gpb}{}$  are thoroughly dissimilar (for example, $S_3 \Wr S_2$ has order $72$ and $\Boxproduct{S_3}{S_2}{}$ has order $2^{\aleph_0}$).\\

In their groundbreaking work on locally compact groups acting on trees \cite{BurgerMozes}, Marc Burger and Shahar Mozes introduce a construction that is now commonly referred to as a {\em universal group with prescribed local action}. Given $F \leq \sym(X)$, a regular tree $T$ whose vertices all have valency $|X|$, and a subgroup $H \leq \aut T$, 
we say that the {\em local action} of $H$ is $F$, or that $H$ is {\em locally-$F$}, if for every vertex $v$ the permutation group induced by the stabiliser $H_v$ on the vertices adjacent to $v$ in $T$ is permutation isomorphic to $F$.

Burger and Mozes' construction (see \cite[Section 3.2]{BurgerMozes}) is a group $U(F) \leq \aut T$ that has local action $F$. Moreover, if $F$ is transitive, then the group $U(F)$ contains a permutationally isomorphic copy of every vertex transitive locally-$F$ subgroup of $\aut T$ (see \cite[Proposition 3.2.2]{BurgerMozes}), and so it is the {\em universal locally-$F$ group}. The group $U(F)$ is
a vertex transitive subgroup of $\aut T$, which is a totally disconnected group under the topology of pointwise convergence. Burger and Mozes' paper only considers the case where $F$ has finite degree (that is, when $\sta$ is finite),
but the above also holds if $F$ is infinite.
If the degree of $F$ is finite then $T$ is locally finite and $U(F)$ is closed and locally compact, and $U(F)$ has a simple subgroup of index $2$ whenever $F$ is transitive and generated by point stabilisers (\cite[Proposition 3.2.1]{BurgerMozes}). These simple subgroups form an important family of examples of compactly generated, totally disconnected, locally compact, simple, non-discrete groups.

To prove the existence of the box product, 
we generalise Burger and Mozes' iconic universal group construction, from regular trees to biregular trees.
We construct a group $\Universal{\gpa}{\gpb}{}$ which acts on a biregular tree $T$ which need not have finite valency.
The action of $\Universal{\gpa}{\gpb}{}$ on $T$ is {\em locally-($\gpa, \gpb$)}; that is, the stabiliser of any vertex $v$ induces either $\gpa$ or $\gpb$ on the neighbours of $v$. When $\gpa$ and $\gpb$ are transitive, the group $\Universal{\gpa}{\gpb}{}$ contains
a permutationally 
isomorphic copy of every other locally-($\gpa, \gpb$) group, and so it is
the {\em universal locally-($\gpa, \gpb$) group}. The groups $U(F)$ then arise as a special case of this construction, with $U(F)$ and $\Universal{F}{\sym(2)}{}$ isomorphic as topological groups
(see Remark~\ref{rem:topo_iso}).
Our results then imply that for $F$ nontrivial and closed (in its permutation topology), $U(F)$ is locally compact (in its permutation topology) if and only if every point stabiliser in $F$ is compact.
Moreover, if these conditions hold, then $U(F)$ is compactly generated whenever $F$ is compactly generated and transitive.\\

The group $\Boxproduct{\gpa}{\gpb}{}$ is the permutation group induced by $\Universal{\gpa}{\gpb}{}$ in its (faithful) action on one of the parts of the bipartition of the vertices of the biregular tree $T$. As topological groups under the permutation topology, $\Boxproduct{\gpa}{\gpb}{}$ is topologically isomorphic to $\Universal{\gpa}{\gpb}{}$. If $\gpa$ and $\gpb$ are also thought of as topological groups under their respective permutation topologies, the box product of $\gpa$ and $\gpb$ inherits topological properties from $\gpa$ but, importantly, it does not inherit discreteness. Its main topological properties 
are established in
Theorems~\ref{theorem:product_as_aut_T}, \ref{theorem:simplicity} and \ref{thm:TopoProperties2}--\ref{thm:discrete}, and can be summarised as follows.

\begin{theorem} \label{thm:summary} Suppose $\gpa \leq \sym(\sta)$ and $\gpb \leq \sym(\stb)$ are permutation groups 
such that $|\sta|, |\stb| > 1$,
with $\gpa$ or $\gpb$ nontrivial. Then the following hold.
\begin{enumerate}
\item
	$\Universal{\gpa}{\gpb}{} \leq \aut T$ is locally-($\gpa$,$\gpb$);
\item
	If $\gpa$ and $\gpb$ are transitive and $H \leq \aut T$ is locally-($\gpa$,$\gpb$), then $H$ is conjugate in $\aut T$ to some subgroup of $\Universal{\gpa}{\gpb}{}$;
\item
	If $\gpa$ and $\gpb$ are closed then $\Universal{\gpa}{\gpb}{}$ is closed;
\item
	If $\gpa$ and $\gpb$ are generated by point stabilisers, then $\Universal{\gpa}{\gpb}{}$ is simple if and only if $\gpa$ or $\gpb$ is transitive;
\item
	If $\gpa$ and $\gpb$ are closed then $\Universal{\gpa}{\gpb}{}$ is locally compact if and only if all point stabilisers in $\gpa$ and $\gpb$ are compact;
\item
	If $\gpa$ and $\gpb$ are closed and compactly generated with compact point stabilisers and only finitely many orbits, and $\gpa$ or $\gpb$ is transitive, then $\Universal{\gpa}{\gpb}{}$ is compactly generated;
\item
	$\Universal{\gpa}{\gpb}{}$ is discrete if and only if $\gpa$ and $\gpb$ are semi-regular.
\end{enumerate}
\end{theorem}

In the final section of the paper, we turn our attention to the class $\mathcal{S}$ of non-discrete, totally disconnected, locally compact, topologically simple 
groups.\footnote{
The definition of $\mathcal{S}$ given here corresponds to that given in \cite{CapraceDeMedts}. Readers should note that some authors use $\mathcal{S}$ to denote the class of {\em compactly generated}, non-discrete, totally disconnected, locally compact, topologically simple groups.
}

We use the box product to construct a family $\mathcal{F}$ of 
abstractly simple groups.
This family contains $2^{\aleph_0}$ pairwise non-isomorphic examples of 
compactly generated groups in $\mathcal{S}$,
and hence gives an affirmative answer to the following question:
{\em 
do there exist uncountably many pairwise non-isomorphic compactly generated groups in $\mathcal{S}$?
}

This question is highlighted in a paper by Pierre-Emmanuel Caprace and Tom De Medts (\cite{CapraceDeMedts}).
The analogous result for discrete groups is due to Ruth Camm (\cite{Cam53}), who proved
in 1953
that there is a continuum of non-isomorphic simple 2-generated groups.

The set of compactly generated groups in $\mathcal{S}$ is known to contain at most $2^{\aleph_0}$ topological group isomorphism types (because all groups in $\mathcal{S}$ are Polish groups), and so it follows that there are precisely $2^{\aleph_0}$ topological group isomorphism types of compactly generated groups in $\mathcal{S}$.

In \cite{CapraceDeMedts}, Pierre-Emmanuel Caprace and Tom De Medts introduce a concept called {\em Lie-reminiscent}. Two groups are {\em locally isomorphic} if they contain isomorphic compact open subgroups, and a group $G$ in $\mathcal{S}$ is {\em Lie-reminiscent} if, for any topologically simple group $H$ in $\mathcal{S}$ which is locally isomorphic to $G$, we have that $H$ is isomorphic to $G$. In \cite{CapraceDeMedts}, the authors find some sufficient conditions for $G \in \mathcal{S}$ to fail to be Lie-reminiscent. This article demonstrates a dramatic failure of the Lie-reminiscent property: $\mathcal{F}$ contains $2^{\aleph_0}$ groups in $\mathcal{S}$ that all share a common compact open subgroup.\\

In Remark~\ref{remark:construction} we outline a general method for using the box product to construct exotic examples of 
compactly generated groups in $\mathcal{S}$.
For example, one can take $\gpa$ to be a finitely generated simple group (with some additional properties) and $\gpb$ to be finite, and
the resulting box product of $\gpa$ and $\gpb$
inherits many of the properties of $\gpa$ but is not discrete. Since the class of finitely generated simple groups is known to contain many unusual groups (see  \cite[Theorem C]{obraztsov96}, \cite[Theorem 28.7]{olshanski} or \cite[Theorem 1.1]{osin10} for example) we see, for the first time, that the class of simple, totally disconnected, locally compact, compactly generated and non-discrete groups is similarly broad.

The results in this paper demonstrate that, for a systematic study of $\mathcal{S}$, the isomorphism relation is too fine. Perhaps the relation of local isomorphism is more suitable.
The following question arises naturally at this point: Do there exist uncountably many local isomorphism types of compactly generated groups in $\mathcal{S}$?

%
%
\section{Preliminaries}
\label{section:preliminaries} 

\subsection{Permutation groups} Let $\VV$ be a non-empty set, and suppose $G$ acts on $\VV$. If $\pta \in \VV$ and $g \in G$, we denote the image of $\pta$ under $g$ by $g\pta$, thus following the convention that our permutations act from the left. This notation extends naturally to sets, so if $\VVsubset \subseteq \VV$ then $g\VVsubset$ is the image $\{g\pta : \pta \in \VVsubset\}$. The {\em setwise stabiliser} of $\VVsubset$ is the group $G_{\{\VVsubset\}}:=\{g \in G : g\VVsubset = \VVsubset\}$, and the {\em pointwise stabiliser} is the group $G_{(\VVsubset)} := \{g \in G : g\pta = \pta,  \forall \pta \in \VVsubset\}$. If $\Lambda$ consists of a single element $\pta$, then the setwise and pointwise stabilisers of $\Lambda$ coincide; this group is called the {\em stabiliser} of $\pta$ and is denoted by $G_{\pta}$. The set $G\pta := \{g\pta : g \in G\}$ denotes the {\em orbit} of $\pta$ (under $G$), and $G$ is said to be {\em transitive} if $\VV$ consists of a single orbit. If $\VVsubset$ is an orbit of $G$, then $G$ induces a subgroup of $\sym(\VVsubset)$ which we denote by $G \big|_{\VVsubset}$. The orbits of point stabilisers in $G$ are called {\em suborbits} of $G$.

The action of $G$ gives rise to a homomorphism from $G$ to the group $\sym(\VV)$ of all permutations of $\VV$. If this homomorphism is injective, then $G$ is said to be acting {\em faithfully}; when this occurs we will often consider $G$ to be a subgroup of $\sym(\VV)$, identifying $G$ with its image in $\sym(\VV)$. Two permutation groups $G \leq \sym(\VV)$ and $H \leq \sym(\VValt)$ are {\em permutation isomorphic} if there exists a bijection $\phi: \VV \rightarrow \VValt$ such that the map $g \mapsto \phi g \phi^{-1}$ is an isomorphism from $G$ to $H$.
The {\em degree} of $G \leq \sym(\VV)$ is $|\VV|$; we say that  $G$ is of {\em nontrivial degree} if $|\VV| > 1$.

A transitive group $G \leq \sym(\VV)$ is {\em primitive} on $\VV$ if the only $G$-invariant equivalence relations on $\VV$ are the trivial relation (each element in $\VV$ is related only to itself) and the universal relation (each element in $\VV$ is related to every element in $\VV$). If $G$ is transitive, then it is primitive if and only if every point-stabiliser $G_{\pta}$ is a maximal subgroup of $G$. A permutation group is {\em semi-regular} if every point stabiliser is trivial; a  {\em regular} permutation group is transitive and semi-regular.

\subsection{Graphs} In this paper a graph $\Gamma$ consists of a set $V\Gamma$ and a set $E\Gamma$ of two-element subsets of $V\Gamma$. The elements in $V\Gamma$ are called the {\em vertices} of $\Gamma$, and the elements of $E\Gamma$ the {\em edges}. The graph is {\em nontrivial} if $E\Gamma$ is non-empty. If two distinct vertices $v$ and $w$ belong to the same edge, they are said to be {\em adjacent}. An {\em arc} in $\Gamma$ is an ordered pair of adjacent vertices, and the set of all arcs is denoted by $A\Gamma$. Thus, our graphs contain no loops or multiple edges, and between any two adjacent 
vertices
there are two arcs, one in each direction. We denote the automorphism group of $\Gamma$ by $\aut \Gamma$.

If $a \in A\Gamma$, we denote by $o(a)$ and $t(a)$ the vertices such that $a = (o(a), t(a))$, and by $\overline{a}$ the arc $(t(a), o(a))$. We will sometimes write the edge $\{o(a), t(a)\}$ as $\{a, \overline{a}\}$. If $v$ is a vertex in $\Gamma$, then $\ArcsFrom(v) := \{a \in A\Gamma : o(a) = v\}$ denotes those arcs originating from $v$, while $\overline{\ArcsFrom}(v) :=  \{a \in A\Gamma : t(a) = v\}$ denotes those arcs that terminate at $v$. We denote the set of vertices adjacent to $v$ by $B(v)$. These notational conventions extend to sets of vertices, so for example if $W \subseteq V\Gamma$ then
$\ArcsFrom(W) := \{a \in A\Gamma : o(a) \in W\}$.

The {\em valency} of $v$ is the cardinal $|B(v)|$; if all valencies are finite the graph is {\em locally finite}. A
{\em graph path}
is a series of distinct vertices $v_0 v_1 \ldots v_n$ such that $v_i \in B(v_{i-1})$ for all integers $i$ satisfying $1 \leq i \leq n$; the {\em length} of this 
graph path is $n$. A graph path
that consists of a single vertex is called {\em trivial} or sometimes {\em empty}. Two vertices are connected if there is a
graph path
between them, and the {\em distance} between two connected vertices $v, w$, which we denote by $d(v, w)$, is the length of the shortest
graph path
between them; if two vertices are not connected then their distance is infinite.
A graph path between vertices $v, w$ that is of length $d(v,w)$ is called a {\em geodesic}.
A graph is {\em connected} if the distance between any two vertices is finite. A {\em cycle} is a series of vertices $v_0 v_1 \ldots v_n v_0$ such that $n > 1$ and $v_0 v_1 \ldots v_n$ and $v_1 \ldots v_n v_0$ are nontrivial
graph paths.
 A {\em tree} is a connected graph that contains no cycles. In a tree $T$, there is a unique
graph path between any two vertices $v, w \in VT$ which we denote by $[v, w]_T$. Note that a graph path in a tree is necessarily a geodesic. If we wish to exclude $w$, we write $[v, w)_T := [v, w]_T \setminus \{w\}$.

If $W$ is a set of vertices in a graph $\Gamma$, we denote by $\Gamma \setminus W$ the subgraph of $\Gamma$ induced on $V\Gamma \setminus W$. On the other hand, if $W$ is a set of edges of $\Gamma$, then $\Gamma \setminus W$ denotes the graph on $V\Gamma$ with edge set $E\Gamma \setminus W$. If $\Gamma$ is connected, and
$\Gamma \setminus \{v\}$ is disconnected for some $v \in V\Gamma$, then $\Gamma$ is said to have {\em connectivity one} and $v$ is called a {\em cut vertex}.
The connected subgraphs of $\Gamma$ that are maximal subject to the condition that they do not have connectivity one are called {\em lobes}.

If $\Gamma$ is vertex transitive, has connectivity one and is not a tree, then every vertex is a cut vertex and there is a tree $T_\Gamma$, called the {\em block-cut-vertex tree} of $\Gamma$, which determines the structure of $\Gamma$. The block-cut-vertex tree is defined as follows. Let $L$ be a set in bijective correspondence with the set of lobes of $\Gamma$. Since $\aut \Gamma$ acts on the set of lobes of $\Gamma$, there is an induced action of $\aut \Gamma$ on $L$. The vertex set of $T_\Gamma$ is the union $L \cup V\Gamma$, and $v \in VT_\Gamma$ and $\ell \in L$ are adjacent in $T_\Gamma$ if and only if $v$ lies in the lobe associated with $\ell$. The action of $\aut \Gamma$ on $V \cup L$ preserves this relation, and so $\aut \Gamma$ acts on $T_\Gamma$. This action is easily seen to be faithful, and we will frequently consider $\aut \Gamma$ to be a subgroup of $\aut T_\Gamma$. Figure~\ref{figure:HEmbeds} shows a graph with connectivity one and its block-cut-vertex tree.

The following result
is a special case of a theorem about directed graphs, and will be useful to us in Section~\ref{section:permutational}.

\begin{theorem}[{\cite[Theorem 2.5]{me:prim_directed_graphs}}]
\label{theorem:MyImprimitiveTreeAction}
Let $G$ be a vertex-transitive group of automorphisms of a connectivity-one graph $\Gamma$, whose lobes have at least three vertices, and let $T$ be the block-cut-vertex tree of $\Gamma$. If there exist distinct vertices
$v, w \in V\Gamma$ such that, for some vertex $x$ in the $T$-geodesic $(v, w)_T$,
\[G_{v, x} = G_{w, x},\]
then $G$ does not act primitively on $V\Gamma$.
\end{theorem}

A {\em ray} (also called a {\em half-line}) in $\Gamma$ is a sequence of distinct vertices $R = \{v_i\}_{i \in \N}$ such that for any $n \in \N$ we have $v_1 v_2 \ldots v_n$ is a
graph path
in $\Gamma$; thus a ray is a one-way infinite
graph path.
The vertex $v_1$ is often called {\em root} of $R$. The {\em ends} of $\Gamma$ are equivalence classes of rays: two rays $R_1$ and $R_2$ lie in the same {\em end} if there is a third ray $R_3$ that contains infinitely many vertices from both $R_1$ and $R_2$. In the special case of an infinite tree $T$, the ends of $T$ are particularly easy to picture: for a fixed vertex $x \in VT$, each end $\epsilon$ contains precisely one ray whose root is $x$, and we denote this ray by $[x, \epsilon)_{T}$.

Let $\cda, \cdb$ be non-zero cardinal numbers. A tree $T$ is {\em $\cda$-regular}, or sometimes {\em regular}, if the valency of every vertex equals $\cda$. There is a natural bipartition of any tree $T$, in which any pair of vertices whose distance is even lie in the same part of the partition. If all vertices in one part
of
this bipartition have valency $\cda$, and all vertices in the other part have valency $\cdb$, we say that $T$ is {\em $(\cda, \cdb)$-biregular}. If $a \in AT$ is an arc in $T$ and both connected components of $T \setminus \{a, \overline{a}\}$ are infinite, then each connected component is called a {\em half-tree} of $T$. We denote the half-tree of $T \setminus  \{a, \overline{a}\}$ containing $o(a)$ by $T_{a}$, so the other connected component of $T \setminus \{a, \overline{a}\}$ is $T_{\overline{a}}$.

Given a permutation group $G \leq \sym(V)$ and distinct elements $v, w \in V$, one can construct a graph whose vertex set is $V$ and whose edges are the elements of the orbit $G\{v, w\}$. Such a graph is called an {\em orbital graph} of $G$.

\subsection{Permutation groups as topological groups}
\label{subsection:perm_topology}

For a thorough introduction see \cite{moller:topgroups} and \cite{woess:topological_groups_and_infinite_graphs}.
If $\VV$ is any non-empty set, then $\sym(\VV)$ can be given a natural topology, that of {\em pointwise convergence}, under which $\sym(\VV)$ is a Hausdorff topological group. A basis of neighbourhoods of the identity is given by point stabilisers of finite subsets of $\VV$, and so an open set in  $\sym(\VV)$ is a union of cosets of point stabilisers of finite subsets of $\VV$. Under this topology, if $G \leq \sym(\VV)$ then a subgroup of $G$ is open in $G$ if and only if it contains the pointwise stabiliser (in $G$) of some finite subset of $\VV$. We will often refer to this topology as the {\em permutation topology}.

For any finite subset $\VVsubset$, the pointwise stabiliser $(\sym(\VV))_{(\VVsubset)}$ is both open and closed in $\sym(\VV)$, so with this topology $\sym(\VV)$ is totally disconnected.
Convergence in $\sym(\VV)$ is natural: a set $W \subseteq \sym(\VV)$ has a limit point $h \in \sym(\VV)$ if and only if, for all finite subsets $\VVsubset \subseteq \VV$, there exists $g \in W$ distinct from $h$ such that $gh^{-1} \in \sym(\VV)_{(\VVsubset)}$. 

Let $G$ be a subgroup of $\sym(\VV)$. The group $G$ is closed if and only if some point stabiliser
is closed, which holds if and only if all point stabilisers in $G$ are closed. The group $G$ is compact if and only if $G$ is closed and all $G$-orbits on $\VV$ are finite. A closed
subgroup 
$G$ is locally compact if and only if $G_{(\VVsubset)}$ is compact for some finite $\VVsubset \subseteq \VV$ (see \cite[Lemma 3.1]{evans97} for example).
In particular, if $\Gamma$ is a connected locally finite graph, then any closed subgroup $G$ of $\aut \Gamma$ will be totally disconnected and locally compact.

A topological space in which every subset is open is called {\em discrete}, and so the permutation topology is discrete on $G$ if and only if there is a finite subset $\VVsubset \subseteq \VV$ such that $G_{(\VVsubset)}$ is trivial. If the permutation topology is discrete on $G$ we say $G$ is a {\em discrete permutation group}.

%
%
\section{Groups acting on trees without inversion: a universal group}
\label{Section:Universal}

Throughout, $\sta$ and $\stb$ will be disjoint sets, each containing at least two elements, and $\gpa \leq \sym(\sta)$ and $\gpb \leq \sym(\stb)$ will be permutation groups.
Let $\T$ denote the $(|\sta|, |\stb|)$-biregular tree. 
Let $\Va$ and $\Vb$ denote the two parts of the natural bipartition of the vertices of $\T$, such that all vertices in $\Va$ have valency $|\sta|$ in $\T$, and all vertices in $\Vb$ have valency $|\stb|$ in $\T$.
A function $\col : A\T \rightarrow \sta \cup \stb$ is called a
{\em legal colouring with $\sta$ and $\stb$}
if it satisfies:
\begin{enumerate}
\item
	for all $v \in \Va$, the restriction $\col \big|_{\ArcsFrom(v)} : \ArcsFrom(v) \rightarrow \sta$ is a bijection;
\item
	for all $v \in \Vb$, the restriction $\col \big|_{\ArcsFrom(v)} : \ArcsFrom(v) \rightarrow \stb$ is a bijection; and
\item \label{def:thirdcondition}
	 for all $v \in VT$, the 
	 map $\col \big|_{\overline{\ArcsFrom}(v)}$ is constant.
\end{enumerate}
One may easily verify that it is always possible to construct a 
legal colouring with $\sta$ and $\stb$.

If $\gpa \leq \sym(\sta)$ and $\gpb \leq \sym(\stb)$, and $\col$ is a
legal colouring with $\sta$ and $\stb$,
define
\begin{equation*} \label{eq:definition}
\Universal{\gpa}{\gpb}{\col} := 
\left \{
	g \in \left (\aut \T \right)_{\{\Va\}} : \col \big|_{\ArcsFrom(gv)}   g \big|_{\ArcsFrom(v)}   \col \big|_{\ArcsFrom(v)}^{-1} \in
	\begin{cases}
		\gpa \text{ for all $v \in \Va$} \\
		\gpb \text{ for all $v \in \Vb$}
	\end{cases}
\right \}.
\end{equation*}
Verifying that $\Universal{\gpa}{\gpb}{\col}$ is a subgroup of $\aut \T$ is tedious but not difficult.
To simplify our exposition we introduce the following notation (which depends on the choice of a legal colouring function $\col$):
\[\ColouringInducedPerm{g}{v} := \col \big|_{\ArcsFrom(gv)}   g \big|_{\ArcsFrom(v)}   \col \big|_{\ArcsFrom(v)}^{-1}.\]

The first construction of the group $\Universal{\gpa}{\gpb}{\col}$ was inspired by \cite[Section 2.2]{me:jls}, and  $\Universal{\gpa}{\gpb}{\col}$ can be constructed using refinements of the arguments in \cite{me:jls} (which use relational structures), so long as $\gpa$ and $\gpb$ are closed (in their respective permutation topologies).
The group acts on a biregular tree that may have infinite valencies, and is 
a generalisation of
Burger and Mozes' universal group $U(F)$, which acts on a
regular tree.

Since $\Universal{\gpa}{\gpb}{\col}$ preserves the parts $\Va$ and $\Vb$, it induces a subgroup of $\sym(\Vb)$. This motivates the following definition.

\begin{definition} The {\em box product of $\gpa$ and $\gpb$}, denoted 
$\Boxproduct{\gpa}{\gpb}{\col}$, is the subgroup of $\sym(\Vb)$ that is induced by $\Universal{\gpa}{\gpb}{\col}$. As we shall see in Proposition~\ref{prop:different_col_are_conjugate}, for two legal colourings $\col, \col'$ 
the groups $\Boxproduct{\gpa}{\gpb}{\col}$ and $\Boxproduct{\gpa}{\gpb}{\col'}$  are permutationally isomorphic, and so we write $\Boxproduct{\gpa}{\gpb}{}$ instead of $\Boxproduct{\gpa}{\gpb}{\col}$ when there is no chance of ambiguity.
\end{definition}

The group $\aut T$ can be viewed as a permutation group in a number of natural ways. Throughout this paper, when $\aut T$ is referred to as a permutation group, we will always mean as a group of permutations of the vertices of $T$. This applies also to $\Universal{\gpa}{\gpb}{\col}$ and to the Burger--Mozes universal group $U(F)$.

\begin{remark} \label{rem:topo_iso} 
As permutation groups, $\Universal{\gpa}{\gpb}{\col}$ and $\Boxproduct{\gpa}{\gpb}{\col}$ 
are not permutation isomorphic in general.
However, when bestowed with their respective permutation topologies, the two groups are isomorphic as topological groups. This is not difficult to see, because for each finite subset $\Phi$ of $VT$, one can find finite subsets $\Phi_\sta \subseteq \Va$ and $\Phi_\stb \subseteq \Vb$ such that the pointwise stabilisers $(\aut T)_{(\Phi_\sta)}$ and $(\aut T)_{(\Phi_\stb)}$ fix $\Phi$ pointwise.

Although $\Universal{\gpa}{\gpb}{\col}$ and $\Boxproduct{\gpa}{\gpb}{\col}$ are isomorphic as topological groups, we will typically use the notation $\Universal{\gpa}{\gpb}{\col}$ in statements concerning the topological properties of the box product, and $\Boxproduct{\gpa}{\gpb}{\col}$ in statements concerning the permutational properties of the box product. This is because the topological properties of the box product are typically symmetric in $\gpa$ and $\gpb$, while the permutational properties are not.

Given a group $F$ of permutations of a set of cardinality
$n \geq 3$,
the Burger--Mozes universal group $U(F)$ acts on the regular tree $T_n$, and this induces an action on $T_{2,n}$, the barycentric subdivision of $T_n$. The universal group  $\Universal{S_2}{F}{}$ also acts on $T_{2,n}$ and (for an appropriate choice of legal colouring $\col$) it is not difficult to see that $U(F)$ and $\Boxproduct{S_2}{F}{\col}$ are permutation isomorphic, where $S_2$ denotes the symmetric group of degree $2$. Hence, if $\sim$ denotes the existence of a topological isomorphism we have: $U(F) \sim \Boxproduct{S_2}{F}{} \sim \Universal{S_2}{F}{} = \Universal{F}{S_2}{} \sim \Boxproduct{F}{S_2}{}$. It follows that our universal construction is a generalisation of the Burger--Mozes universal group.
\end{remark} 

\begin{remark} It is not difficult to see that
 $\Universal{\gpa}{\gpb}{\col}$ enjoys the following monotonicity property: if $\gpa_1 \leq \gpa_2$ and $\gpb_1 \leq \gpb_2$, then $\Universal{\gpa_1}{\gpb_1}{\col} \leq \Universal{\gpa_2}{\gpb_2}{\col}$ for any legal colouring $\col$.

At this juncture it might be instructive for the reader to consider that if $\gpa$ and $\gpb$ are both transitive, then $\Universal{ \langle 1 \rangle }{\gpb}{\col} \leq \Universal{\gpa}{\gpb}{\col}$ is transitive on $\Va$ and $\Universal{\gpa}{ \langle 1 \rangle }{\col} \leq \Universal{\gpa}{\gpb}{\col}$ is transitive on $\Vb$. In fact the orbits of $\Universal{\gpa}{\gpb}{\col}$ are completely determined by the orbits of its local actions $\gpa$ and $\gpb$, as we shall see in Proposition~\ref{prop:products_orbits}.
\end{remark}

Let us say that a group $H \leq \aut \T$ is {\em locally-$(\gpa, \gpb)$} if $H$ fixes
setwise the parts $\Va$ and $\Vb$,
and for all vertices $v$ of $\T$ the group $H_v \big|_{B(v)} \leq \sym (B(v))$ induced by the vertex stabiliser $H_v$ is permutation isomorphic to $\gpa$ if $v \in \Va$ and $\gpb$ if $v \in \Vb$. As a subgroup of $\aut \T$, the group $\Universal{\gpa}{\gpb}{\col}$ has the following properties.

\begin{theorem} \label{theorem:product_as_aut_T} Suppose $\sta$, $\stb$ are disjoint sets of cardinality at least two, and $T$ is the $(|\sta|, |\stb|)$-biregular tree. Given permutation groups $\gpa \leq \sym (\sta)$ and $\gpb \leq \sym (\stb)$, and a legal colouring $\col$ of $\T$,
\begin{enumerate}
\item \label{item:orbits_in_theorem:product_as_aut_T}
	$\Delta \subseteq \sta \cup \stb$ is an orbit of $\gpa$ or $\gpb$ if and only if $t \left ( \col^{-1}\Delta \right )$ is an orbit of $\Universal{\gpa}{\gpb}{\col}$;
\item
	given legal colourings $\col, \altcol$, the groups $\Universal{\gpa}{\gpb}{\col}$ and $\Universal{\gpa}{\gpb}{\altcol}$ are conjugate in $\aut \T$;
\item \label{item:universal}
	if $\gpa$ and $\gpb$ are transitive, and $H \leq \aut \T$ is locally-($\gpa, \gpb$), then $H$ is edge-transitive and for some legal colouring $\col$ we have
	\[H \leq \Universal{\gpa}{\gpb}{\col};\]
\item
	$\Universal{\gpa}{\gpb}{\col}$ is locally-($\gpa, \gpb$);
\item
	if $\gpa$ and $\gpb$ are closed, then $\Universal{\gpa}{\gpb}{\col}$ is a closed subgroup of $\aut \T$.
\end{enumerate}
\end{theorem}

Thus, when $\gpa$ and $\gpb$ are transitive, $\Universal{\gpa}{\gpb}{\col}$ is the universal locally-($\gpa, \gpb$) group, in that it contains a permutationally isomorphic copy of every locally-($\gpa, \gpb$) group acting on the biregular tree $\T$.
Theorem~\ref{theorem:product_as_aut_T} follows from Lemma~\ref{lem:g_and_two_colourings}, Propositions~\ref{prop:products_orbits}--\ref{prop:H_embeds_in_some_product} and Lemmas~\ref{lem:product_locally_MN}--\ref{lem:product_is_closed}.\\

Our arguments in this section rely on manipulating sequences of bijections and restrictions of functions.
Although these manipulations can look rather daunting, they are in fact very simple and involve only three basic steps which are given in the following lemma.

\begin{lemma} \label{lemma:ColourMapContinuous}
If $g, h \in (\aut T)_{\{\Va\}}$, then for all $w \in VT$,
\begin{enumerate}
\item \label{ItemBasicStepOne}
	$g \big|_{\ArcsFrom(w)}^{-1} = g^{-1} \big|_{\ArcsFrom(gw)}$; 
\item \label{ItemBasicStepTwo}
	$g \big|_{\ArcsFrom(hw)} h \big|_{\ArcsFrom(w)} g^{-1} \big|_{\ArcsFrom(gw)} = \left ( g h g^{-1} \right ) \big|_{\ArcsFrom(gw)}$; and
\item \label{ItemBasicStepThree}
	if $\col = \altcol  g$, then $\col \big|_{\ArcsFrom(w)} = \altcol \big|_{\ArcsFrom(gw)}   g \big|_{\ArcsFrom(w)}$ and  $\col \big|_{\ArcsFrom(w)}^{-1} = g^{-1} \big|_{\ArcsFrom(gw)} \altcol \big|_{\ArcsFrom(gw)}^{-1}$. 
\end{enumerate}
Moreover, for a fixed $w \in VT$, the map $g \mapsto \ColouringInducedPerm{g}{w}$ is continuous.
\end{lemma}

\begin{proof} Statements (\ref{ItemBasicStepOne})--(\ref{ItemBasicStepThree}) are obviously true. Suppose $w \in \Va$ and let $G := (\aut T)_{\{\Va\}}$. We will show that the map $\chi_w : G \rightarrow \sym(\sta)$ given by $g \mapsto \ColouringInducedPerm{g}{w}$ is continuous, where $G$ and $\sym(\sta)$ have their respective permutation topologies.

Clearly $\chi_w$ is surjective. Fix some open set $W \subseteq \sym(\sta)$ and choose any $g$ in the preimage $\chi_w^{-1}(W)$. Then $\chi_w(g)$ lies in the open set $W$ so there exists some finite subset $Z \subseteq \sta$ such that $\chi_w(g) \sym(\sta)_{(Z)} \subseteq W$. Since $\col \big|^{-1}_{\ArcsFrom(w)} (Z)$ is a finite set of arcs, we can find a finite set $\Phi$ of vertices such that $w \in \Phi$ and $G_{(\Phi)}$ fixes individually each arc in $\col \big|^{-1}_{\ArcsFrom(w)} (Z)$.

If $h \in G_{(\Phi)}$, then $\ArcsFrom(ghw) = \ArcsFrom(gw)$ and $\ArcsFrom(hw) = \ArcsFrom(w)$, and so:
\[
	\chi_w(gh) =
	\col \big|_{\ArcsFrom(gw)} g \big|_{\ArcsFrom(w)} \col \big|_{\ArcsFrom(w)}^{-1}
	\col \big|_{\ArcsFrom(w)} h \big|_{\ArcsFrom(w)} \col \big|_{\ArcsFrom(w)}^{-1}
	= \chi_w(g)\chi_w(h).
\]
Furthermore, one may easily verify that $\chi_w(h)$ fixes $Z$ pointwise.
Hence $\chi_w(gG_{(\Phi)}) \subseteq \chi_w(g) \sym(\sta)_{(Z)} \subseteq W$. Now $gG_{(\Phi)}$ contains $g$ and is open in $G$, and we have shown that $gG_{(\Phi)} \subseteq \chi_w^{-1}(W)$. Hence $\chi_w^{-1}(W)$ is open and $\chi_w$ must therefore be continuous.

If $w \in \Vb$, then a symmetric argument shows that the map $\chi_w : G \rightarrow \sym(\stb)$ is continuous.
\end{proof}

\begin{remark} \label{remark:DefnOfUinTermsOfTheta}
For a given legal colouring $\col$, the continuous maps $\chi_w : (\aut T)_{\{\Va\}} \rightarrow \sym(\sta)$ and $\chi_w : (\aut T)_{\{\Va\}} \rightarrow \sym(\stb)$ given by $g \mapsto \ColouringInducedPerm{g}{w}$ for, respectively, $w \in \Va$ and $w \in \Vb$,  provide an alternate description of $\Universal{\gpa}{\gpb}{\col}$  as the intersection over all vertices $w \in VT$ of the preimages of $\gpa$ and $\gpb$ under the maps $\ColouringInducedPerm{\cdot}{w}$. That is,
\[\Universal{\gpa}{\gpb}{\col} = \left ( \bigcap_{w \in \Va} \chi^{-1}_w(\gpa) \right ) \cap \left ( \bigcap_{w \in \Vb} \chi_w^{-1}(\gpb) \right ).\]
\end{remark}

In what follows, $\col$ and $\col'$ are arbitrary
legal colourings with $\sta$ and $\stb$.
To simplify notation, we  associate each constant function $\col \big|_{\overline{\ArcsFrom}(w)}$ with its image in $\sta \cup \stb$, so for all $w \in VT$ we have $\col \big|_{\overline{\ArcsFrom}(w)} \in \sta \cup \stb$. 
We now present a lemma that underpins most results in this section.

\begin{lemma} \label{lem:g_and_two_colourings} If the distance between $v, v' \in VT$ is even, then there exists $\sigma \in  \sym(\sta \cup \stb)_{\{\sta\}}$ satisfying $\sigma \altcol \big|_{\overline{\ArcsFrom}(v')} = \col \big|_{\overline{\ArcsFrom}(v)}$, and for all such $\sigma$ there exists a unique automorphism $g \in (\aut T)_{\{\Va\}}$ such that $g v = v'$ and $\col = \sigma \altcol g$.
\end{lemma}
\begin{proof} 
Since $\col$ and $\altcol$ are legal colourings, we can always find $\sigma \in  \sym(\sta \cup \stb)_{\{\sta\}}$ satisfying $\sigma \altcol \big|_{\overline{\ArcsFrom}(v')} = \col \big|_{\overline{\ArcsFrom}(v)}$. For each integer $n \geq 0$, let $B_n$ (resp. $B'_n$) be the subtree of $T$ induced by those vertices whose distance in $T$ from $v$ (resp. $v'$) is at most $n$. Because $d(v, v')$ is even, $v$ and $v'$ belong to the same part of the bipartition of $T$.

Let $g_0 : B_0 \rightarrow B'_0$ be the map taking $v$ to $v'$. Our proof that an appropriate element $g \in \aut T$ exists relies on an easy induction argument with the following induction hypothesis: for all integers $n \geq 1$, there exists a graph isomorphism $g_n : B_n \rightarrow B_n'$ such that $\sigma \altcol \big|_{B'_n} g_n = \col \big|_{B_n}$ and $g_n \big|_{B_{n-1}} = g_{n-1}$. The base case is true because $\altcol \big|_{\ArcsFrom(v')}^{-1} \sigma^{-1} \col \big|_{\ArcsFrom(v)}$ is a bijection from $\ArcsFrom(v)$ to $\ArcsFrom(v')$, and so it induces a graph isomorphism $g_1: B_1 \rightarrow B_1'$ with $g_1 v  = v'$.

Suppose, for some $n \geq 1$, that our induction hypothesis holds.
If $w$ is a vertex in $B_n \setminus B_{n-1}$, then there is a unique arc $a$ in $B_n$ whose origin is $w$, and $g_n a$ is the unique arc in $B'_n$ whose origin is $g_n w$.  Now $\col^{'} \big|^{-1}_{\ArcsFrom(g_n w)}  \sigma^{-1} \col \big|_{\ArcsFrom(w)}$ maps $a$ to $g_n a$ and so it
induces a bijection
$h_w : B(w) \setminus VB_n \rightarrow B(g_n w) \setminus VB'_n$. Let $g_{n+1}$ be $g_n$ extended by all these $h_w$, for $w \in B_n \setminus B_{n-1}$, and note that $g_{n+1}$ is a graph isomorphism from $B_{n+1}$ to $B'_{n+1}$ with $g_{n+1} \big|_{B_n} = g_n$.
It remains to show that $\sigma \altcol \big|_{B'_{n+1}} g_{n+1} = \col \big|_{B_{n+1}}$. We need only consider arcs in $B_{n+1} \setminus B_n$. If $a$ is such an arc, then either $t(a) \in VB_n$ or $o(a) \in VB_n$. If $t(a) \in VB_n$, chose an arc $b$ in $B_n$ with $t(b) = t(a)$. Since $t(b) = t(a)$ and $t(g_{n+1} a) = t(g_{n+1} b)$ we have $\col b = \col a$ and $\altcol   g_{n+1} b =  \altcol   g_{n+1} a$. Moreover,  $\sigma \altcol   g_{n+1} b = \col b$ by the induction hypothesis, and so $\col a = \col b = \sigma \altcol   g_{n+1} b = \sigma \altcol   g_{n+1} a$. On the other hand, if $w:=o(a) \in VB_n$ then $g_{n+1} w = g_n w$ and $g_{n+1} t(a) = h_{w} t(a)$. Hence (by definition of $h_w$) we have
$\sigma \altcol   g_{n+1} a = \col a$.
Thus, our induction hypothesis is true for all integers $n \geq 1$.

We now construct an element $g \in (\aut T)_{\{\Va\}}$ that satisfies $\sigma \altcol g = \col$ and $gv = v'$.
For each vertex $w \in VT \setminus \{v\}$ there is a unique $n(x) \in \N$ such that $w \in B_{n(x)} \setminus B_{n(x)-1}$. Define $g : VT \rightarrow VT$ to be such that $gv := v'$ and for all $w \in VT \setminus \{v\}$ set $gx := g_{n(x)}x$. One may easily verify that $g \in (\aut T)_{\{\Vb\}}$.
Moreover, for all $a = (w, w') \in AT$, there exists an $m$ such that $a \in AB_{m}$. Then $ga = (g_{n(w)} w, g_{n(w')} w') = (g_m w, g_m w') = g_m a$. Hence $\sigma \col' g a = \sigma \col' g_{m} a = \col a$.

It remains to prove that $g$ is unique. Suppose $h \in \aut \T$ with  $\sigma \altcol   h = \col$ and $hv = v'$. We claim that if $gh^{-1}w = w$ for some $w \in VT$, then $gh^{-1}$ fixes $B(w)$ pointwise. Indeed, if $a \in \ArcsFrom(w)$ then $gh^{-1} a \in \ArcsFrom(w)$, and  $\sigma \altcol g h^{-1} a = \col h^{-1} a = \sigma \altcol h h^{-1} a = \sigma \altcol a$. Our claim follows immediately because $\altcol \big|_{\ArcsFrom(w)}$ is a bijection. Since $hv = v'$, we have $gh^{-1}v' = v'$, and our claim tells us that $gh^{-1}$ fixes $B(v')$ pointwise. Since $T$ is connected, $gh^{-1}$ fixes $VT$ pointwise, and is therefore the identity.
\end{proof}

The orbits of $\Universal{\gpa}{\gpb}{\col}$ relate naturally to those of $\gpa$ and $\gpb$.

\begin{proposition}\label{prop:products_orbits} Vertices $v, v'$ in $\Va$ (resp. $\Vb$) lie in the same orbit of  $\Universal{\gpa}{\gpb}{\col}$ if and only if $\col \big|_{\overline{\ArcsFrom}(v)}$ and $\col \big|_{\overline{\ArcsFrom}(v')}$ lie in the same orbit of $\gpb$ (resp. $\gpa$).
\end{proposition}

\begin{proof} Suppose $v, v' \in \Va$ (a symmetric argument holds for $v, v' \in \Vb$). If 
there exists $\sigma \in \gpb$ such that $\sigma \col \big|_{\overline{\ArcsFrom}(v)} = \col \big|_{\overline{\ArcsFrom}(v')}$, let $\hat{\sigma}$ be the element in $\sym(\sta \cup \stb)_{\{\stb\}}$ that is $\sigma$ on $\stb$ and trivial on $\sta$. By Lemma~\ref{lem:g_and_two_colourings}, there exists a unique $g \in (\aut \T)_{\{\Va\}}$ such that $gv = v'$ and $\col = \hat{\sigma} \col g$. If $w \in \Va$, then
$\ColouringInducedPerm{g}{w}$
$= \left ( \col  g \right ) \big|_{\ArcsFrom(w)} \col \big|^{-1}_{\ArcsFrom(w)} = \left ( \hat{\sigma}^{-1} \col \right ) \big|_{\ArcsFrom(w)} \col \big|^{-1}_{\ArcsFrom(w)} = \hat{\sigma}^{-1} \big|_{\sta} \in \gpa$,
and if 
$w \in \Vb$ then 
a similar argument shows
$\ColouringInducedPerm{g}{w} \in \gpb$.
Hence $g \in \Universal{\gpa}{\gpb}{\col}$.

Conversely, if $gv = v'$ for some $g \in \Universal{\gpa}{\gpb}{\col}$, fix $(w, v) \in \overline{\ArcsFrom}(v)$ and define
$\sigma := \ColouringInducedPerm{g}{w} \in \gpb$.
Then $\sigma \col (w, v) = \col (gw, v')$. By the definition of $\col$, we have $\col \big|_{\overline{\ArcsFrom}(v)} = \col (w, v)$ and $\col \big|_{\overline{\ArcsFrom}(v')} = \col (gw, v')$. Hence $\sigma \col \big|_{\overline{\ArcsFrom}(v)} = \col \big|_{\overline{\ArcsFrom}(v')}$.
\end{proof}

Any two legal colourings give rise to essentially the same permutation group.

\begin{proposition} \label{prop:different_col_are_conjugate} The groups $\Universal{\gpa}{\gpb}{\col}$ and $\Universal{\gpa}{\gpb}{\altcol}$ are conjugate in $\aut \T$.
\end{proposition}
\begin{proof} Choose vertices $v, v' \in V\T$ such that $\col \big|_{\overline{\ArcsFrom}(v)}$ and $\altcol \big|_{\overline{\ArcsFrom}(v')}$ are equal. By Lemma~\ref{lem:g_and_two_colourings}, there is a unique element $g \in (\aut \T)_{\{\Va\}}$ such that $gv = v'$ and $\col = \altcol   g$. Choose $h \in (\aut \T)_{\{\Va\}}$, and set $h' := g h g^{-1}$.  If $w' \in V\T$, then $w:= g^{-1} w'$ and $w'$ must both lie in $\Va$, or both lie in $\Vb$, and
\begin{align*}
\altcol \big|_{\ArcsFrom(h' w')} h' \big|_{\ArcsFrom(w')} \altcol \big|_{\ArcsFrom(w')}^{-1}
	& = \altcol \big|_{\ArcsFrom(ghw)} g \big|_{\ArcsFrom(hw)} h \big|_{\ArcsFrom(w)} g^{-1} \big|_{\ArcsFrom(gw)} \altcol \big|_{\ArcsFrom(gw)}^{-1}\\
	&= \col \big|_{\ArcsFrom(hw)} h \big|_{\ArcsFrom(w)} \col \big|_{\ArcsFrom(w)}^{-1}.
\end{align*}
Hence $h' \in \Universal{\gpa}{\gpb}{\altcol}$ if and only if $h \in \Universal{\gpa}{\gpb}{\col}$.
\end{proof}

If $\gpa$ and $\gpb$ are transitive, then $\Universal{\gpa}{\gpb}{\col}$ contains an isomorphic copy of every locally-$(\gpa, \gpb)$ group.

\begin{proposition} \label{prop:H_embeds_in_some_product} If $\gpa$ and $\gpb $ are transitive, and $H \leq \aut \T$ is locally-$(\gpa, \gpb)$, then $H \leq \Universal{\gpa}{\gpb}{\col}$ for some legal colouring $\col$ of $\T$.
\end{proposition}
\begin{proof} Suppose $\gpa$ and $\gpb $ are transitive, and $H \leq \aut \T$ is locally-$(\gpa, \gpb)$, so $H$ has precisely two orbits on $VT$, acting transitively on $\Va$ and $\Vb$. Fix adjacent vertices $\va \in \Va$ and $\vb \in \Vb$. Since $H_{\va} \big|_{B(\va)}$ is permutation isomorphic to $\gpa$ there exists a bijection $\phi : B(\va) \rightarrow \sta$ such that $\phi H_{\va} \big|_{B(\va)} \phi^{-1} = \gpa$. Similarly, there exists a bijection $\psi: B(\vb) \rightarrow \stb$ such that $\psi H_{\vb} \big|_{B(\vb)} \psi^{-1} = \gpb$. For each integer $n \geq 0$, let $S_n$ be the set of vertices in $\T$ whose distance from $\va$ is precisely $n$. For $v \in \bigcup_{i\geq 2}S_i$, there is a unique 
graph path
$[\va, v]_T$ in $T$ between $\va$ and $v$; denote the unique vertex in $B(v) \cap [\va, v]_T$ by $v_{1}$ and the unique vertex in $B(v_{1}) \cap [\va, v_1]_T$ by $v_{2}$.

We inductively construct a subset $J := \{h_v : v \in VT\}$ of $H$ that we will use to define a legal colouring $\col$ such that $H \leq \Universal{\gpa}{\gpb}{\col}$. Each element $h_v \in J$ will have the property that $h_v \va = v$ if $v \in \Va$ and $h_v \vb = v$ if $v \in \Vb$.

Let $h_{\va}$ be the identity, and for each $v \in S_1$, choose $h_v \in H_{\va}$ such that $h_v \vb = v$. Notice that for all $v \in \cup_{i=0}^1S_i$, $h_v \va = v$ if $v \in \Va$ and $h_v \vb = v$ if $v \in \Vb$. Let $n \in \N$, and suppose, for each $v \in \bigcup_{i=0}^n S_{i}$, we have chosen an element $h_v \in H$ such that $h_v \va = v$ if $v \in \Va$ and $h_v \vb = v$ if $v \in \Vb$. For $v \in S_{n+1}$, choose $h \in H_{v_1}$ such that $h v_2 = v$ (which we can do because $H_{v_1}$ is transitive on $B(v_1)$) and set $h_v:= h h_{v_2}$. If $v \in \Va$, then $v_2 \in \Va$ and $h_v \va =  v$, and  similarly if $v \in \Vb$ then $h_v \vb = v$. Note that for all $v \in \bigcup_{i\geq 2} S_i$ we have $h_{v_2} h^{-1}_v = h \in H_{v_1}$, and for all $v, w \in S_1$ we have $h_w, h_v \in H_{\va}$, so $h_w^{-1}h_v \in H_{\va}$.

We now describe a colouring $\col$, and show that it is a legal colouring. Define $\col: AT: \rightarrow \sta \cup \stb$ as follows:
for $v \in \Va$ let $\col \big|_{\ArcsFrom(v)} : \ArcsFrom(v) \rightarrow \sta$ be $a \mapsto \phi h_v^{-1} t(a)$, and for $v \in \Vb$ let $\col \big|_{\ArcsFrom(v)} : \ArcsFrom(v) \rightarrow \stb$ be $a \mapsto \psi h_v^{-1} t(a)$.
It is clear that for all $v \in V\T$ the map $\col \big|_{\ArcsFrom(v)}$ is a bijection, so it remains to show that condition (\ref{def:thirdcondition}) of the definition holds; that is, the image of $\col \big|_{\overline{\ArcsFrom}(v)}$ has cardinality one for all $v \in VT$. We show this first for $\va$, then for all $v \in S_1$, and finally for all $v \in \bigcup_{n \geq 2} S_n$.

If $(w, \va)$ is any arc in $\overline{\ArcsFrom}(\va)$, then $w \in \Vb$ and $\col \big|_{\overline{\ArcsFrom}(\va)}(w, \va) = \col \big|_{\ArcsFrom(w)}(w, \va) = \psi h_w^{-1}\va = \psi \va$ (because $w \in S_1$ so $h_w \in H_{\va}$). Therefore the cardinality of the image of $\col \big|_{\overline{\ArcsFrom}(\va)}$ is one.

If $v \in S_1$ and $
(w, v) \in \overline{\ArcsFrom}(v)$, then either $w = \va$, or $w \in S_2 \subseteq \Va$. If $w = \va$ then $\col \big|_{\overline{\ArcsFrom}(v)}(w, v) = \phi v$. On the other hand, if $w \in S_2$ then $w_1 = v$ and $w_2 = \va$. Hence $\col \big|_{\overline{\ArcsFrom}(v)}(w, v) = \col \big|_{\ArcsFrom(w)}(w, v) = \phi h^{-1}_w v$ and $h_w \in H_{w_1}  h_{w_2} = H_v h_{\va} = H_v$. Thus $\col \big|_{\overline{\ArcsFrom}(v)}(w, v) = \phi v$. Therefore, if $v \in S_1$ then the cardinality of the image of $\col \big|_{\overline{\ArcsFrom}(v)}$ is always one. 

Suppose that $v \in S_n$, for some $n\geq 2$.
Choose $(w, v) \in \overline{\ArcsFrom}(v)$ such that $w \not = v_1$; then $v = w_1$ and $v_1 = w_2$. If $v \in \Va$, then $w \in \Vb$ and $\col \big|_{\overline{\ArcsFrom}(v)}(w, v) = \col \big|_{\ArcsFrom(w)}(w, v) = \psi h_w^{-1} v = \psi h_w^{-1} w_1$. Since $h_w^{-1} \in h_{w_2}^{-1} H_{w_1}$, we have $\col \big|_{\overline{\ArcsFrom}(v)}(w, v) = \psi h^{-1}_{w_2} w_1 = \psi h_{v_1}^{-1} v =  \col \big|_{\overline{\ArcsFrom}(v)}(v_1, v)$. Hence, the cardinality of the image of $\col \big|_{\overline{\ArcsFrom}(v)}$ is one. A symmetric argument shows that the same is true if $v$ is instead chosen from $\Vb$. We have thus demonstrated that $\col$ is a legal colouring.

Finally, we show $H \leq \Universal{\gpa}{\gpb}{\col}$. Choose $g \in H$ and $v \in V\T$.
First, let us suppose that $v \in \Va$. Then $gv \in \Va$, and by construction, $h_{v} \va = v$ and $h_{gv}\va = gv$. Hence $h_{gv}^{-1} g h_v \in H_{\va}$, and therefore $\sigma := (h_{gv}^{-1} g h_v) \big|_{B(\va)} \in H_{\va} \big|_{B(\va)}$. Now $\sigma = h_{gv}^{-1} \big|_{B(g h_v \va)} g \big|_{B(h_v \va)} h_v \big|_{B(\va)} = h_{gv} \big|_{B(\va)}^{-1} g \big|_{B(v)} h_v \big|_{B(\va)}$; therefore
\begin{equation} \label{eq:Sigma}
\sigma = 
h_{gv} \big|^{-1}_{B(\va)} t \big|_{A(gv)}
g \big|_{A(v)}
t \big|^{-1}_{A(v)}     h_v \big|_{B(\va)}.
\end{equation}
Moreover, by our choice of legal colouring we have
\begin{equation} \label{eq:ColouringOne}
	\col \big|_{A(v)} =
\phi \left ( h_v^{-1} \big|_{B(h_v \va)} \right ) t \big|_{A(v)} = \phi \left (  h_v \big|_{B(\va)}^{-1} \right ) t \big|_{A(v)},
\end{equation}
and
\begin{equation} \label{eq:ColouringTwo}
	\col \big|_{A(gv)} = \phi \left ( h_{gv}^{-1} \big|_{B(h_{gv} \va)} \right ) t \big|_{A(gv)} = \phi \left ( h_{gv} \big|_{B(\va)}^{-1} \right ) t \big|_{A(gv)}.
\end{equation}
Using (\ref{eq:Sigma}), (\ref{eq:ColouringOne}) and (\ref{eq:ColouringTwo}) one can easily verify that
$\ColouringInducedPerm{g}{v}$
$= \phi \sigma \phi^{-1}$. Hence $\ColouringInducedPerm{g}{v} \in \phi H_{\va}\big|_{B(\va)} \phi^{-1} = \gpa$.
Now, let us suppose instead that $v \in \Vb$. A similar argument gives us that $gv \in \Vb$ and $h_{gv}^{-1} g h_v \in H_{\vb}$, so
$\ColouringInducedPerm{g}{v}$
$\in \psi H_{\vb} \big|_{B(\vb)} \psi^{-1} = \gpb$. Therefore $g \in \Universal{\gpa}{\gpb}{\col}$.
\end{proof}

In the proposition above, the requirement that $\gpa$ and $\gpb$ be transitive is essential. For example, if $G$ is the automorphism group of the graph $\Gamma$ pictured in Figure~\ref{figure:HEmbeds}, then $G$ is transitive on the vertices of $\Gamma$ and induces a faithful action on its block-cut-vertex tree $T_\Gamma$. In its action on $T_\Gamma$, it is easily seen that the group $G$ has two orbits: one orbit consists of the $T_\Gamma$-vertices corresponding to the vertices of $\Gamma$, and the other consists of those $T_\Gamma$-vertices that correspond to the lobes of $\Gamma$. Choose any pair $v$ and $w$ of adjacent vertices in $T_\Gamma$, and let $\gpa$ denote $G_v \big|_{B(v)}$ and $\gpb$ denote $G_w \big|_{B(w)}$. If it were possible to find a legal colouring $\col$ of $T_\Gamma$ such that $G \leq \Universal{\gpa}{\gpb}{\col}$, then this would contradict Proposition~\ref{prop:products_orbits}, because neither $\gpa$ nor $\gpb$ are transitive.
Another example is the following: if $C_2$ denotes the cyclic group of order $2$ acting on three points, and $T_3$ denotes the trivalent tree, then for any end $\xi$ of $T_3$ the end stabiliser $(\aut T_3)_\xi$ has $C_2$ as its local action but $(\aut T_3)_\xi$ is clearly not conjugate to a subgroup of the Burger--Mozes group $U(C_2) \leq \aut T_3$.

\begin{figure}[h]  
\centering
\includegraphics[width=120mm]{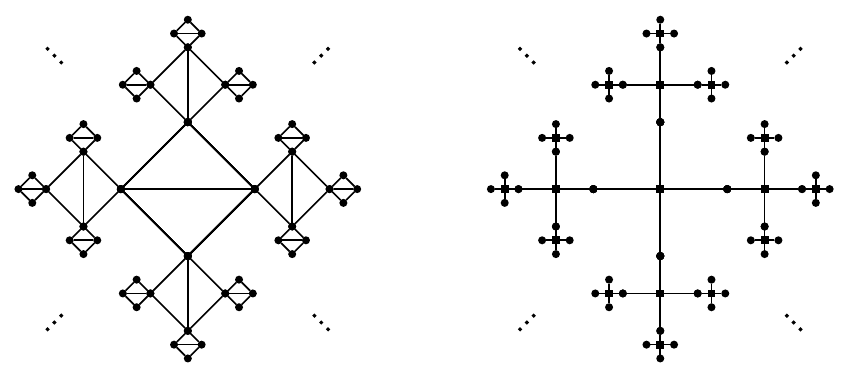}
\caption{An infinite vertex-transitive graph $\Gamma$ (left) and its block-cut-vertex tree $T_\Gamma$ (right)}
\label{figure:HEmbeds}
\end{figure}

\begin{lemma} \label{lem:product_locally_MN} $\Universal{\gpa}{\gpb}{\col}$ is locally-$(\gpa, \gpb)$.
\end{lemma}

\begin{proof} Fix $v \in V\T$.
If $v \in \Va$ (resp. $v \in \Vb$) we claim the
 map $g \big|_{\ArcsFrom(v)} \mapsto \ColouringInducedPerm{g}{v}$
 is an isomorphism from $\left ( \Universal{\gpa}{\gpb}{\col}  \right )_v \big|_{\ArcsFrom(v)}$ to $\gpa$
(resp. $\gpb$). We prove this only for $v \in \Va$ since a similar argument works for $v \in \Vb$.

 It is clear this map is a monomorphism. Fix $\sigma \in \gpa$, and let $\hat{\sigma} \in \sym(\sta \cup \stb)$ be the permutation that equals $\sigma$ on $\sta$ and is the identity on $\stb$. Now $\hat{\sigma} \col$ is a legal colouring, and Lemma~\ref{lem:g_and_two_colourings} guarantees that there exists a unique element $g \in (\aut \T)_{\{\Va\}}$ such that $g v = v$ and $\col g = \hat{\sigma} \col$.
Hence, for all $w \in V\T$ we have  
$\ColouringInducedPerm{g}{w}$
$=  \hat{\sigma} \big|_{\col (\ArcsFrom(w))}$, and from this it follows that $g \in \left ( \Universal{\gpa}{\gpb}{\col}  \right )_v$ and
$\ColouringInducedPerm{g}{v} = \sigma$.
Finally, we observe that the permutation groups induced by $\left ( \Universal{\gpa}{\gpb}{\col}  \right )_v$ on $\ArcsFrom(v)$ and $B(v)$ are permutation isomorphic.
\end{proof}

\begin{lemma} \label{lem:product_is_closed} If $\gpa$ is a closed subgroup of $\sym(\sta)$ and $\gpb$ is a closed subgroup of $\sym(\stb)$, then $\Universal{\gpa}{\gpb}{\col}$ is a closed subgroup of $\aut \T$.
\end{lemma}

\begin{proof}
One may easily verify that $A:=(\aut T)_{\{\Va\}}$ is a closed subgroup of $\aut T$. Suppose $\gpa$ and $\gpb$ are closed. The two maps $\chi_w : (\aut T)_{\{\Va\}} \rightarrow \sym(\sta)$ and $\chi_w : (\aut T)_{\{\Va\}} \rightarrow \sym(\stb)$ given by $g \mapsto \ColouringInducedPerm{g}{w}$ for, respectively, $w \in \Va$ and $w \in \Vb$, are continuous by Lemma~\ref{lemma:ColourMapContinuous}. Therefore the preimages $\chi_w^{-1}(\gpa)$ for $w \in \Va$ and $\chi_w^{-1}(\gpb)$ for $w \in \Vb$ are closed. In Remark~\ref{remark:DefnOfUinTermsOfTheta}, it is noted that $\Universal{\gpa}{\gpb}{\col}$ can be written as the intersection over all vertices $w \in VT$ of these preimages. Thus $\Universal{\gpa}{\gpb}{\col}$, being the intersection of closed sets, is closed.
\end{proof}

We next prove that $\Universal{\gpa}{\gpb}{\col}$ has a subgroup isomorphic to $\gpa$ and a subgroup isomorphic to $\gpb$. These subgroups are contained in point stabilisers in $\Universal{\gpa}{\gpb}{\col}$. 

If we are given $\mu \in \gpa$, let $\hat{\mu} \in \sym(\sta \cup \stb)$ be such that $\hat{\mu} \big|_\sta = \mu$ and $\hat{\mu} \big|_\stb = 1_{\gpb}$. By Lemma~\ref{lem:g_and_two_colourings}, for each vertex $v \in \Va$ there is a unique automorphism
$g_{\mu, v} \in (\aut T)_{\{\Va\}}$ 
such that $g_{\mu, v} v = v$ and $\col = \hat{\mu} \col g_{\mu, v}$. One may quickly verify that $g_{\mu, v} \in \Universal{\gpa}{\gpb}{\col}$. For $v \in \Vb$ and $\mu \in \gpb$ define $g_{\mu, v}$ similarly.

\begin{proposition} \label{iso_subgroups} If $v \in \Va$ and $w \in \Vb$, then $\hat{\gpa}(v) := \{g_{\mu, v} : \mu \in \gpa\}$ is a subgroup of $\Universal{\gpa}{\gpb}{\col}_v$ and $\hat{\gpb}(w) := \{g_{\tau, w} : \tau \in \gpb\}$ is a subgroup of $\Universal{\gpa}{\gpb}{\col}_w$.
Moreover, $\hat{\gpa}(v)$ is isomorphic to $\gpa$, and $\hat{\gpb}(w)$ is isomorphic to $\gpb$.
\end{proposition}

\begin{proof} Suppose $v \in \Va$. If $\mu, \tau \in \gpa$, then $g_{\mu, v} g_{\tau, v} v = v = g_{\tau \mu, v} v$ and $\hat{\tau} \hat{\mu} \col g_{\tau \mu, v} = \col = \hat{\tau} (\hat{\mu} \col g_{\mu, v}) g_{\tau, v}$. Since $\hat{\tau} \hat{\mu} $ fixes $\stb$ pointwise, we have $\hat{\tau} \hat{\mu} \col \big|_{\overline{A}(v)} = \col \big|_{\overline{A}(v)}$. It follows immediately from Lemma~\ref{lem:g_and_two_colourings} that $g_{\tau \mu, v} = g_{\mu, v} g_{\tau, v}$. Similarly, it follows from Lemma~\ref{lem:g_and_two_colourings} that if $e$ denotes the identity in $\gpa$, then $g_{e, v} = 1 \in \Universal{\gpa}{\gpb}{\col}$. Finally, we have $g_{\mu, v} g_{\mu^{-1}, v} = g_{\mu^{-1}\mu, v} = g_{e, v} = 1 \in \Universal{\gpa}{\gpb}{\col}$, so $g_{\mu, v}^{-1} = g_{\mu^{-1}, v}$. Hence $\hat{\gpa}(v) \leq \Universal{\gpa}{\gpb}{\col}_v$.  A similar argument holds for $w \in \Vb$ and $\hat{\gpb}(w)$. 

Since $\mu, \tau \in \gpa$, we have $g_{\mu, v} = g_{\tau, v}$ if and only if $\mu = \tau$. Indeed, if $\mu = \tau$, then $\hat{\tau} \col g_{\tau, v} = \col = \hat{\mu} \col g_{\mu, v} = \hat{\tau} \col g_{\mu, v}$, and so $g_{\tau, v} = g_{\mu, v}$ by Lemma~\ref{lem:g_and_two_colourings}. On the other hand, if $g_{\tau, v} = g_{\mu, v}$ then $1 = g_{\mu, v} g_{\tau, v}^{-1} = g_{\tau^{-1} \mu, v}$. Hence $\col = \hat{\tau}^{-1} \hat{\mu} \col g_{\tau^{-1}\mu, v} = \hat{\tau}^{-1} \hat{\mu} \col$, so $\tau$ and $\mu$ must be equal.

Let $\varphi : \hat{\gpa}(v) \rightarrow M$ be given by $g_{\mu, v} \mapsto \mu^{-1}$. This map is obviously surjective, and we have shown it to be well-defined and injective.
For $\mu, \tau \in \gpa$ we have $g_{\mu, v} g_{\tau, v} = g_{\tau \mu, v}$, so $\varphi$ is a homomorphism. A symmetric argument shows that $\hat{\gpb}(w)$ is isomorphic to $\gpb$.
\end{proof}

We now give two
lemmas,
presented without proof, which will be used in the following sections. The first lemma is an immediate consequence of Proposition~\ref{prop:products_orbits} and Lemma~\ref{lem:product_locally_MN}.

\begin{lemma} \label{lemma:edge_orbits} Suppose $v_1, v_2 \in \Va$ and $w_1, w_2 \in \Vb$. Edges $\{v_1, w_1\}$ and $\{v_2, w_2\}$ in $\T$ lie in the same orbit of $\Universal{\gpa}{\gpb}{\col}$ if and only if $v_1, v_2$ lie in the same orbit of $\Universal{\gpa}{\gpb}{\col}$ and $w_1, w_2$ lie in the same orbit of $\Universal{\gpa}{\gpb}{\col}$. \qed
\end{lemma}

Recall that $\T_a$ and $\T_{\overline{a}}$ are the two half-trees of $\T \setminus \{a, \overline{a}\}$. To prove the following lemma, simply check that the obvious choice for $g \in \aut T$ lies in $\Universal{\gpa}{\gpb}{\col}$.

\begin{lemma} \label{lem:prod_has_independence} If $a$ is any arc in $\T$, then for all $h \in \Universal{\gpa}{\gpb}{\col}_{a}$ there exists $g \in \Universal{\gpa}{\gpb}{\col}$ such that $g$ fixes $\T_{\overline{a}}$ pointwise and $g \big|_{\T_a} = h \big|_{\T_a}$. \qed
\end{lemma}

For a permutation group $H \leq \sym(\Omega)$ together with some $H$-invariant subset $\Delta \subseteq \Omega$, it is not true in general that the quotient group $H/H_{(\Delta)}$, viewed as a topological group under the quotient topology, is topologically isomorphic to the induced permutation group $H \big|_\Delta$ under its permutation topology. However, for our universal groups we have the following.

\begin{proposition} \label{prop:quotient_iso_to_M} Suppose $\gpa$ and $\gpb$ are closed, with $G:= \Universal{\gpa}{\gpb}{\col}$. If $v \in \Va$ (resp. $v \in \Vb$) then $G_v/G_{(B(v))}$ under the quotient topology is topologically isomorphic to $\gpa$ (resp. $\gpb$) under its permutation topology.
\end{proposition}
\begin{proof} We show that $H := G_v \big|_{B(v)}$ is topologically isomorphic to $G_v/G_{(B(v))}$; the result then follows from Lemma~\ref{lem:product_locally_MN}. 
If $\pi : G_v \rightarrow G_v/G_{(B(v))}$ denotes the quotient map, then the map $\Theta : G_v/G_{(B(v))} \rightarrow H$ given by $\pi(g) \mapsto g \big|_{B(v)}$ is an abstract isomorphism, and it is a routine exercise to check that $\Theta$ is continuous.
To prove that $\Theta^{-1}$ is continuous, it is sufficient to show that for all finite sets $F \subseteq VT$, there is a finite set $F' \subseteq B(v)$ such that $(G_v)_{(F)} \big|_{B(v)}  = (G_v \big|_{B(v)})_{(F')}$.

Given $F \subseteq VT$, let $F'$ be the (possibly empty) set consisting of all vertices in $B(v)$ that lie on a geodesic between $v$ and some vertex in $F$. Let $C$ be the connected component of $T \setminus F'$ containing $v$. Clearly, $(G_v)_{(F)} \big|_{B(v)} \leq (G_v \big|_{B(v)})_{(F')}$. On the other hand, if $g'\big|_{B(v)} \in (G_v \big|_{B(v)})_{(F')}$,
then $g'$ fixes $C$ setwise, and so there exists an automorphism $g$ of $T$ such that $g$ and $g'$ agree on $C$ and $g$ fixes $T\setminus C$ pointwise. One may easily verify that $g \in (G_v)_{(F)}$ and $g \big|_{B(v)} = g' \big|_{B(v)}$.
\end{proof}

%
%
\section{Simplicity}
\label{section:Simplicity}

In his influential paper \cite{tits}, Jacques Tits introduced the following {\em property (P)}, which is sometimes known as {\em Tits' independence property}. Let $G$ act on a (not necessarily locally-finite) tree $T$. If $\mathcal{P}$ is a non-empty finite or infinite 
graph path
in $T$, for each vertex $v$ in $T$ there is a unique vertex $\pi_{\mathcal{P}}(v)$ in $\mathcal{P}$ that is closest to $v$. This gives rise to a well-defined map on $VT$, in which each vertex $v$ is mapped to $\pi_{\mathcal{P}}(v)$. For each vertex $q$ in $\mathcal{P}$, the set $\pi^{-1}_{\mathcal{P}}(q)$ of vertices in $T$ that are mapped to $q$ by $\pi_{\mathcal{P}}$ is the vertex set of a subtree of $T$. The pointwise stabiliser $G_{(\mathcal{P})}$ of $\mathcal{P}$ leaves each of these subtrees invariant, and so we can define $G^q_{(\mathcal{P})}$ to be the subgroup of $\sym(\pi^{-1}_{\mathcal{P}}(q))$ induced by $G_{(\mathcal{P})}$. Thus, we have homomorphisms $\varphi_q: G_{(\mathcal{P})} \rightarrow G_{(\mathcal{P})}^q$ for each $q \in V\mathcal{P}$ from which we obtain the natural homomorphism,
\begin{equation} \label{eq:propertyP}
\varphi: G_{(\mathcal{P})} \rightarrow \prod_{q \in V\mathcal{P}} G_{(\mathcal{P})}^q.
\end{equation}
The group $G$ is said to have property (P) if the homomorphism (\ref{eq:propertyP}) is an isomorphism for every possible choice of $\mathcal{P}$. Intuitively, property (P) means that $G_{(\mathcal{P})}$ acts independently on each of the subtrees branching from $\mathcal{P}$.

\begin{theorem}[{\cite[Th\'{e}or\`{e}me 4.5]{tits}}] \label{theorem:tits} Suppose $T$ is a
tree. If $G \leq \aut T$ satisfies property (P), no proper non-empty subtree of $T$ is invariant under $G$, and no end of $T$ is fixed by $G$, then the group
$G^+ := \langle G_{(v, w)} : \{v, w\} \in ET \rangle$
is simple (and possibly trivial).
\qed
\end{theorem}

\begin{theorem} \label{thm:prod_has_prop_P} Let $\sta$ and $\stb$ be finite or infinite sets whose cardinality is at least two. Suppose $\gpa \leq \sym(\sta)$ and $\gpb \leq \sym(\stb)$ are two permutation groups, and let $\T$ be the $(|\sta|, |\stb|)$-biregular tree. If $\col$ is a
legal colouring with $\sta$ and $\stb$, 
then $\Universal{\gpa}{\gpb}{\col}$ satisfies Tits' independence property (P).
\end{theorem}
\begin{proof} Let $\Path$ be any non-empty
graph path
in $\T$ and let $G$ denote $\Universal{\gpa}{\gpb}{\col}$. Suppose we are given  $g \in G_{(\Path)}$ and $p \in V\Path$. Note that $B(p)$ contains either one or two vertices in $\Path$, so we may write $B(p) \cap V\Path = \{q, q'\}$, where $q, q'$ could be equal. By Lemma~\ref{lem:prod_has_independence} there exists $g' \in G$ such that $g'$ fixes $T_{(q, p)}$ pointwise and agrees with $g$ on $T_{(p, q)}$. The element $g'$ fixes $\Path$ pointwise, and so in particular it lies in $G_{(p, q')}$, and so we may apply the lemma again to deduce the existence of $g_p \in G$ such that $g_p$ fixes $T_{(q', p)}$ pointwise and agrees with $g'$ on $T_{(p, q')}$. The element $g_p$ therefore fixes $V\T \setminus \pi^{-1}(p)$ pointwise, agrees with $g$ on $\pi^{-1}(p)$, and satisfies $g_p \big|_{\pi^{-1}(p)} \in G_{(\Path)}^p$.

Now $g = \prod_{p \in V\Path} g_p$, and so the map given by $g \mapsto \prod_{p \in V\Path} g_p \big|_{\pi^{-1}(p)}$ is a homomorphism from $G_{(\Path)}$ to $\prod_{p \in V\Path} G_{(\Path)}^p$. It is easily verified that this map is an isomorphism, from which it follows that $G$ has property (P).
\end{proof}

There are many mild conditions under which $\Universal{\gpa}{\gpb}{\col}$ leaves invariant no proper non-empty subtree of $T$ and fixes
no end of $T$. For example: $\gpa$ and $\gpb$ contain no fixed points.\\

We will presently cite a result in
Jean-Pierre Serre's
book \cite{serre:trees}, but we must do so with care since Serre's definition of a graph differs from ours in that loops (arcs $a$ that satisfy $o(a) = t(a)$) and multiple arcs (distinct arcs $a, a'$ that satisfy $o(a) = o(a')$ and $t(a) = t(a')$) are permitted. Trees in \cite{serre:trees} cannot contain loops nor multiple arcs, and so we may use the term tree without ambiguity (see \cite[pp. 13--17]{serre:trees}).

A group $G$ acts on a tree $T$ {\em without inversion} if there is no edge $\{v, w\} \in T$ for which $(v, w)$ lies in the orbit $G(w, v)$. Suppose $T$ is a tree and $G$ acts on $T$ without inversion. Following \cite[pp. 25]{serre:trees}, $G \backslash T$ consists of a set of vertices and a set of edges. The vertices (resp. edges) are the orbits of $G$ on the vertices (resp. edges) of $T$. Notions like adjacency and arcs extend naturally to $G \backslash T$. Note that $G \backslash T$ may not be a graph (according to our definition), since there may be more than one edge between two given vertices.

\begin{theorem}[{\cite[Corollary 1 in Section 5.4]{serre:trees}}] \label{thm:serre} Suppose $T$ is a tree, and let $G$ be a group acting on $T$ without inversion. Let $R$ be the group generated by all vertex stabilisers $G_v$, $v \in VT$. Then $R$ is a normal subgroup of $G$, and $G/R$ is isomorphic to the fundamental group of $G \backslash T$. \qed
\end{theorem}

It is an immediate consequence of this theorem that $G = R$ if and only if $G \backslash T$ is a tree
which holds if and only if $G$ is generated by the vertex groups of $G \backslash T$
(see \cite[Exercise 2 in Section 5.4]{serre:trees}, for example).

As before, let $\sta$ and $\stb$ be finite or infinite sets with cardinality at least two, $\gpa \leq \sym(\sta)$ and $\gpb \leq \sym(\stb)$, denote the $(|\sta|, |\stb|)$-biregular tree by $T$ and choose some legal colouring $\col$.
Our next lemma follows from Proposition~\ref{prop:products_orbits} and Lemma~\ref{lemma:edge_orbits}. In particular, we have that $\Universal{\gpa}{\gpb}{\col} \backslash \T$ contains no loops or multiple edges.

\begin{lemma} \label{lem:QuotientIsAGraph} If $\cda$ (resp. $\cdb$) denotes the number of orbits of $\gpa$ (resp. $\gpb$), then $\Universal{\gpa}{\gpb}{\col} \backslash \T$ is the complete bipartite graph $K_{m, n}$. \qed
\end{lemma}

If $\gpa$ and $\gpb$ are transitive, then Lemma~\ref{lem:QuotientIsAGraph} and \cite[Section 4.1, Theorem 6]{serre:trees} imply that $G:= \Universal{\gpa}{\gpb}{\col}$ has an amalgamated free product structure,
\[G = G_u \ast_{G_{(u, w)}}  G_v,\]
where $u$ and $v$ are any two adjacent vertices in $\T$.\\

Recall that $\Boxproduct{\gpa}{\gpb}{}$ is the subgroup of $\sym(\Vb)$ that is induced by
 $\Universal{\gpa}{\gpb}{}$ acting on $V\T$,
and $\Boxproduct{\gpa}{\gpb}{} \cong \Universal{\gpa}{\gpb}{}$ as topological groups.
Following \cite{mollervonk}, we say a permutation group $G \leq \sym(\VV)$ is {\em generated by point stabilisers} if
$G = \langle G_v : v \in \VV \rangle$.

\begin{theorem} \label{theorem:simplicity} Suppose $\gpa$ and $\gpb$ are permutation groups of (not necessarily finite) degree at least two, both groups are generated by point stabilisers and at least one group is nontrivial. Then
$\Universal{\gpa}{\gpb}{}$
is simple if and only if $\gpa$ or $\gpb$ is transitive.
\end{theorem}
\begin{proof} Suppose $M$ is a subgroup of $\sym(\sta)$ and $N$ is a subgroup of $\sym(\stb)$; let $\T$
and $\col$ be
as above.
Write $G:= \Universal{\gpa}{\gpb}{\col}$,
recall $G^+ := \langle G_{(v, w)} : \{v, w\} \in ET \rangle$,
and denote the stabiliser of $v \in VT$ in $G^+$ by $G_v^+$.

We claim that $G_v^+ = G_v$ for all $v \in V\T$. Given $v \in V\T$, it follows from Lemma~\ref{lem:product_locally_MN} that the group $G_v \big|_{B(v)}$ is generated by point stabilisers. Moreover, $G_{(v, w)} \big|_{B(v)} \leq G_v^+ \big|_{B(v)}$ for all $w \in B(v)$.
Hence
$G_v^+ \big|_{B(v)} = G_v \big|_{B(v)}$. In particular, $G_v^+$ and $G_v$ both have the same orbits on $B(v)$. Since $G_{(v, w)} \leq G_v^+$ for all $w \in B(v)$, our claim follows.

Hence $R := \langle G_v : v \in V\T \rangle \leq G^+$. Since $\gpa$ or $\gpb$ is nontrivial, $R$ is nontrivial. Moreover, by Theorem~\ref{thm:serre}, $R$ is a normal subgroup of $G$, and $G=R$ if and only if $G \backslash \T$ is a tree.

If $\gpa$ and $\gpb$ are intransitive, then $G \backslash \T$ is not a tree by Lemma~\ref{lem:QuotientIsAGraph}, and so $R$ is a nontrivial proper normal subgroup of $G$. 

Conversely, suppose $\gpa$ or $\gpb$ is transitive. Then $G \backslash \T$ is a tree, and so $G = R = G^+$. Let $\epsilon$ be an end in $\T$, and suppose $T'$ is some non-empty proper subtree of $\T$.  Now $G_v \big|_{B(v)}$ is transitive for all $v$ in one part of the bipartition of $\T$, and so we may choose such a vertex $v$ from $V\T \setminus VT'$.
Clearly $G_v$ does not fix $\epsilon$, nor does $G_v$ leave $T'$ invariant. Thus, by
Theorems~\ref{theorem:tits} and \ref{thm:prod_has_prop_P},
$G = G^+$ is simple.
\end{proof}

Since point stabilisers are maximal in primitive permutation groups, all non-regular primitive permutation groups are generated by any two distinct point stabilisers.

\begin{corollary} If $\gpa$ and $\gpb$ are non-regular primitive permutation groups, then
$\Universal{\gpa}{\gpb}{}$
is simple. \qed
\end{corollary}

\begin{remark} For a nontrivial permutation group $\gpa\leq \sym(\sta)$, the Burger--Mozes group $U(\gpa)$ is permutation isomorphic to $\Boxproduct{S_2}{\gpa}{}$ (as noted in Remark~\ref{rem:topo_iso}), and so both groups can be thought of as groups of automorphisms of the $|\sta|$-regular tree $T$. Recall that for $H \leq \aut T$ the group $H^+$ is generated by all the pointwise stabilisers of edges in $T$, and it is easy to see that $H^+ \unlhd H$. Of course $(\Boxproduct{S_2}{\gpa}{})^+$ and $U(\gpa)^+$ are also isomorphic as permutation groups.

If $\rm{Id}$ denotes the trivial subgroup of $S_2$, and $\gpa$ is transitive and generated by point stabilisers, then Theorem~\ref{theorem:simplicity} implies that
$\Boxproduct{\rm{Id}}{\gpa}{}$ is simple, and hence $(\Boxproduct{\rm{Id}}{\gpa}{})^+ = \Boxproduct{\rm{Id}}{\gpa}{}$. For each edge $\{v,w\} \in ET$ we have $(\Boxproduct{S_2}{\gpa}{})_{(v,w)} \leq \Boxproduct{\rm{Id}}{\gpa}{}$, and so it follows that $U(\gpa)^+ \cong (\Boxproduct{S_2}{\gpa}{})^+ = \Boxproduct{\rm{Id}}{\gpa}{}$. Hence $U(\gpa)^+$ is simple when $\gpa$ is nontrivial, transitive and generated by point stabilisers. This fact was already known for finite $\gpa$ (see \cite[Proposition 3.2.1]{BurgerMozes}).
\end{remark}

%
%
\section{Permutational properties}
\label{section:permutational}

Recall the famous primitivity conditions for the unrestricted wreath product $\gpa \Wr_{\stb} \gpb$ in its product action on the set $\sta^{\stb}$ of functions from $\stb$ to $\sta$ (see \cite[Lemma 2.7A]{dixon&mortimer} for example).
\begin{enumerate}
\item
	$\gpa \Wr_{\stb} \gpb$ is transitive if and only if $\gpa$ is transitive; and
\item
	$\gpa \Wr_{\stb} \gpb$ is primitive if and only if $\gpa$ is primitive but not regular, and $\gpb$ is transitive and finite.
\end{enumerate}

Because of
these
criteria, the wreath product can be used to easily build new primitive groups from other primitive groups. It features prominently in the seminal O'Nan--Scott Theorem, which classifies the finite primitive permutation groups (see \cite{liebeck&praeger&saxl:finite_onan_scott}, for example).

Compare the primitivity criteria for the wreath product above with the following
remarkable
result for the box product.

\begin{theorem} \label{thm:PermutationalProperties} Given 
permutation groups $\gpa \leq \sym (\sta)$ and $\gpb \leq \sym (\stb)$ of nontrivial degree,
the permutation group $\Boxproduct{\gpa}{\gpb}{} \leq \sym(\Vb)$ satisfies:
\begin{enumerate}
\item \label{item:product_trans_iff_M_is_trans}
	$\Boxproduct{\gpa}{\gpb}{}$ is transitive if and only if $\gpa$ is transitive; and
\item
	$\Boxproduct{\gpa}{\gpb}{}$ is primitive if and only if $\gpa$ is primitive but not regular, and $\gpb$ is transitive.
\end{enumerate}
\end{theorem}

Despite the striking similarity between these two sets of conditions, $\gpa \Wr_{\stb} \gpb$ and $\Boxproduct{\gpa}{\gpb}{}$ distort the actions of $\gpa$ and $\gpb$ in opposite ways. This is most apparent when $\gpa$ is subdegree-finite and primitive but not regular and $\gpb$ is 
finite
and transitive. All nontrivial orbital graphs of $\Boxproduct{\gpa}{\gpb}{}$ are tree-like: they are locally finite, connected and have infinitely many ends; on the other hand nontrivial orbital graphs of $\gpa \Wr_{\stb} \gpb \leq \sym(\sta^{\stb})$ are locally finite and connected but with at most
one end
(see \cite[Theorem 2.4]{me:jls}). Figure \ref{fig:orbital_graph_vs} shows (left) an orbital graph of $\Boxproduct{S_3}{S_2}{}$ (which is necessarily infinite) and (right) an orbital graph of $S_3 \Wr S_2$ (which is necessarily finite).

\begin{figure}[h]
\centering
\begin{minipage}{.5\textwidth}
  \centering
  \includegraphics[width=.7\linewidth]{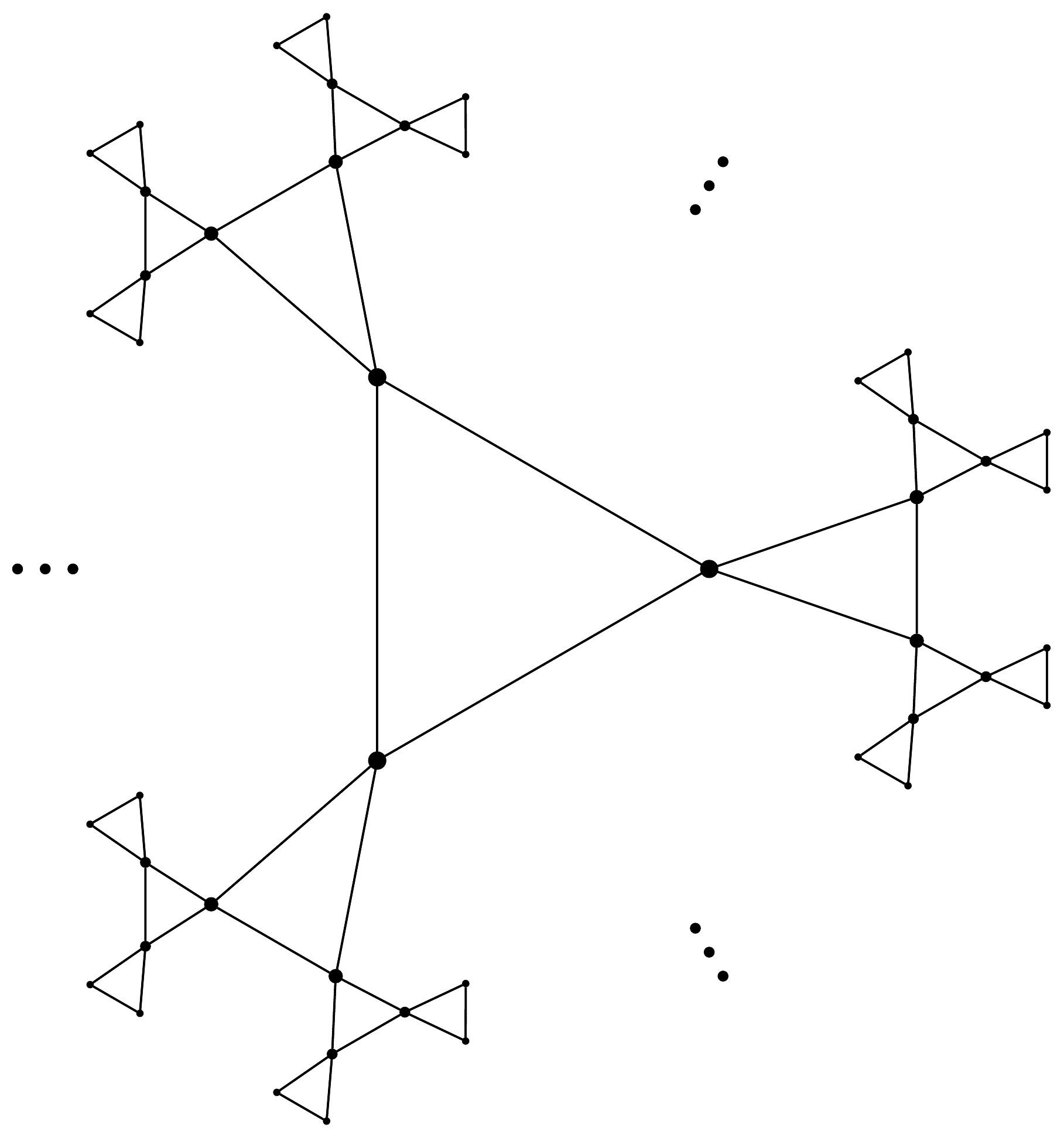}
\end{minipage}%
\begin{minipage}{.5\textwidth}
  \centering
  \includegraphics[width=.4\linewidth]{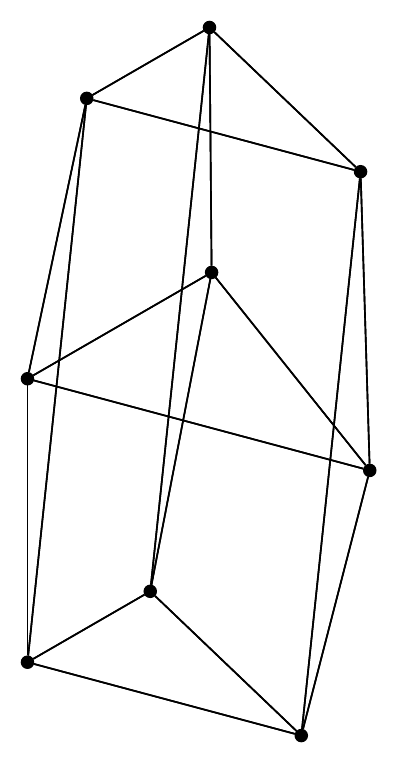}
\end{minipage}
\caption{An orbital graph of $\Boxproduct{S_3}{S_2}{}$ (left) and of $S_3 \Wr S_2$ (right)}
\label{fig:orbital_graph_vs}
\end{figure}

\begin{proof} [Proof of Theorem~\ref{thm:PermutationalProperties}] Part (\ref{item:product_trans_iff_M_is_trans}) is Theorem~\ref{theorem:product_as_aut_T}~(\ref{item:orbits_in_theorem:product_as_aut_T}).
Write $G:=\Universal{\gpa}{\gpb}{\col}$ for some
legal colouring with $\sta$ and $\stb$,
let $\T$ be the $(|\sta|,|\stb|)$-biregular tree.

We show first that if $G$ is primitive on $\Vb$, then $\gpa$ is primitive but not regular and $\gpb$ is transitive. Suppose that $\gpb$ is not transitive on $\stb$.
Fix $v \in \Va$ and distinct vertices $w, w' \in B(v) \subseteq \Vb$. Choose $v' \in B(w)$ such that $\col (w, v)$ and $\col (w, v')$ lie in distinct orbits of $\gpb$. By Proposition~\ref{prop:products_orbits}, the vertex $v'$ does not lie in the orbit $Gv$. Let $\Gamma$ be the orbital graph whose vertex set is $\Vb$ and edge set is the orbit $G\{w, w'\}$. Notice that if two vertices in $\Gamma$ are adjacent, then their distance in $\T$ is two. Therefore, if $x \in \Vb$ lies in the connected component of $\T \setminus \{w\}$ that contains $v'$, then any
graph path
in $\Gamma$ from $x$ to $w$ must contain an edge (in $\Gamma$) between $w$ and some vertex $w'' \in B(v')$. But this implies that $\{w, w''\}$ lies in the orbit $G\{w, w' \}$, which requires that $v' \in Gv$, and we already have that $v' \not \in Gv$. Hence $G$ has a nontrivial orbital graph on $\Vb$ that is not connected. The connected components of this graph give rise to a $G$-invariant equivalence relation on $\Vb$, and so the action of $G$ on $\Vb$ is not primitive. 

Suppose $\gpa$ is not primitive on $\sta$. If $\gpa$ is not transitive then $G$ is not transitive (and therefore not primitive) on $\Vb$ by Proposition~\ref{prop:products_orbits}. Suppose then that $\gpa$ is imprimitive. By D.~G.~Higman's Theorem (\cite[1.12]{higman}), there exists an orbital graph of $\gpa$ that is nontrivial and not connected.
Choose a vertex $v \in \Va$. By Lemma~\ref{lem:product_locally_MN}, 
some nontrivial orbital graph $\Delta$ of $G_v \big|_{B(v)}$ on $B(v)$ is not connected; let $\{w, w'\}$ be an edge in $\Delta$. Now take $\Gamma$ to be the orbital graph whose vertex set is $\Vb$ and whose edge set is the orbit $G\{w, w'\}$. We claim that $\Gamma$ is not connected. Indeed, any pair of adjacent vertices in $\Gamma$ are at distance $2$ in the tree $\T$. Therefore any
graph path in $\Gamma$
between distinct vertices in $B(v)$ contains only vertices in $B(v) = V\Delta$. It follows that $\Gamma$ is not connected, and thus that $G$ is not primitive on $\Vb$.

Suppose that $\gpa$ is regular. Let  $\Gamma$ be a new graph on $\Vb$, in which vertices are adjacent if and only if their distance in $\T$ is two. Thus $\Gamma$ is has connectivity one. If $|\sta| = 2$, then $\Gamma$ is a regular tree and the bipartition of $\Gamma$ is a system of imprimitivity for $G$ on $\Vb$ so the action of $G$ on $\Vb$ is not primitive. If $|\sta| \geq 3$, then the lobes of $\Gamma$ contain at least three vertices, $G \big|_{\Vb} \leq \aut \Gamma$ acts vertex transitively, and the block-cut-vertex tree of $\Gamma$ is $\T$. Choose any vertex $v \in \Va$, and distinct vertices $w, w' \in B(v) \subseteq \Vb$. Because $\gpa$ is regular, $G_{w, v} = G_{w', v}$.
By Theorem~\ref{theorem:MyImprimitiveTreeAction}, the action of $G$ on $V\Gamma = \Vb$ is not primitive.

We have shown that if $G$ is primitive on $\Vb$, then $\gpa$ is primitive but not regular and $\gpb$ is transitive; let us turn our attention to the converse. Suppose that $\gpa$ is primitive and not regular, and $\gpb$ is transitive. Let $\sim$ be a nontrivial $G$-invariant equivalence relation on $\Vb$. Choose distinct $w, w' \in \Vb$ with $w \sim w'$. Let $v$ denote the vertex adjacent to $w'$ in the 
geodesic
$[w, w']_{\T}$, and let $w''$ denote the vertex adjacent to $v$ in $[w, w')_{\T}$. Since $G_v \big|_{B(v)}$ is permutation isomorphic to $\gpa$, it is primitive and not regular on $B(v)$, and so there is an element $h \in G_{v, w''}$ that does not fix $w'$.
By Lemma~\ref{lem:prod_has_independence}, there exists $g \in G$ such that $g$ fixes the half-tree $\T_{(w'', v)}$ pointwise (so in particular $g$ fixes $w$) and 
$gw' = hw'$.
Hence $w' \sim w = gw \sim gw'$, and $d_{\T}(w', gw') = 2$. However, $G_v \big|_{B(v)}$ is primitive, so $\sim$ must be a universal relation on $B(v)$. Since $\gpb$ is transitive, 
it follows that $\sim$ is universal on $\Vb$. Hence there are no proper nontrivial $G$-invariant equivalence relations on $\Vb$, so the action of $G$ on $\Vb$ is primitive.
\end{proof}

Recall that a permutation group is {\em subdegree-finite} if any point stabiliser has only orbits of finite length.

\begin{proposition} \label{prop:finite_subdegrees} Suppose $\gpa$ and $\gpb$ are 
permutation groups of nontrivial degree.
Then $\Boxproduct{\gpa}{\gpb}{}$ is subdegree-finite if and only if $\gpa$ is subdegree-finite and all orbits of $\gpb$ are finite.
\end{proposition}
\begin{proof} Let $G:=\Universal{\gpa}{\gpb}{\col}$. Suppose $\gpa$ is subdegree-finite and all orbits of $\gpb$ are finite. Choose distinct $w, w' \in \Vb$ and let $w_0 w_1 \cdots w_n$ denote the
geodesic
$[w, w']_T$. Since $G_{w, w'} = G_{w_0, \ldots, w_n}$ we have 
\[|G_w w'| = |G_{w_0} : G_{w_0, \ldots, w_n}| \leq |G_{w_0} w_1| \cdot \prod_{i=1}^{n-1} |G_{w_{i-1}, w_i} w_{i+1}|. \]
Now $|G_{w_0} w_1|$ is finite because $\gpb$ has only finite obits, and each $|G_{w_{i-1}, w_i} w_{i+1}|$ is finite because both $\gpa$ and $\gpb$ are subdegree-finite. Hence the action of $G$ on $\Vb$ is subdegree-finite.

On the other hand, if $\gpa$ is not subdegree-finite, then for all $v \in \Va$ there exist $w, w' \in B(v) \subseteq \Vb$ such that $G_{w, v}w'$ is infinite. Hence $G_w w'$ is not finite. If $\gpb$ has an infinite orbit, then for all $w \in \Vb$ there exists $v \in B(w) \subseteq \Va$ such that $G_w v$ is infinite. Hence for any $w' \in B(v) \setminus \{w\}$, the orbit $G_w w'$ is infinite.
\end{proof}

%
%
\section{Topological properties of the product}

If we bestow $\sym(\sta)$ and $\sym(\stb)$ with their respective permutation topologies, then 
the box product of $\gpa$ and $\gpb$ (under its usual permutation topology)
 preserves some topological properties of $\gpa$ and $\gpb$ but does not preserve discreteness. In this section, all topological statements are with respect to the permutation topology, which is described in Section~\ref{subsection:perm_topology}. 

\begin{lemma} \label{lem:closure} \label{lem:tree_locally_compact}   Let $T$ be a tree with no vertex of valence one,
and consider $\aut T$ as a topological group under its permutation topology.
Suppose $G \leq \aut T$ is closed and fixes setwise the two parts $\VV_1, \VV_2$ of the bipartition of $T$. Then $G \big|_{\VV_i}$ is closed in $\sym(\VV_i)$ for $i = 1, 2$, and if $G_v \big|_{B(v)}$ is closed for all $v \in VT$ then the following are equivalent:
\begin{enumerate}
\item \label{item1:lem:tree_locally_compact}
	for all $v \in \VV_1$ and $w \in \VV_2$, all point stabilisers in $G_v \big|_{B(v)}$ are compact and $G_w \big|_{B(w)}$ is compact; 
\item \label{item2:lem:tree_locally_compact}
	all point stabilisers in $G\big|_{\VV_2}$ are compact.
\end{enumerate}
\end{lemma}

\begin{proof}
Let $A:= (\aut T)_{\{\VV_1\}}$ and let $\varphi : A \rightarrow A \big|_{\VV_1}$ be the map taking $g \in A$ to $g \big|_{\VV_1} \in A \big|_{\VV_1}$. Since the tree $T$ contains no vertices of valency one, the kernel of this map is trivial and $\varphi$ is an isomorphism. 
Moreover, $\varphi$ is a topological isomorphism. This is not difficult to see, because for each finite subset $\Phi \subseteq V\T$, one can find finite subsets $\Phi_1 \subseteq \VV_1$ and $\Phi_2 \subseteq \VV_2$ such that the pointwise stabilisers $A_{(\Phi_1)}$ and $A_{(\Phi_2)}$ fix $\Phi$ pointwise.
It follows that $G \Big|_{\VV_1}$ is closed in $\sym(\VV_1)$. A symmetric argument shows that $G \Big|_{\VV_2}$ is closed in $\sym(\VV_2)$.

Now suppose $G_v \big|_{B(v)}$ is closed for all $v \in VT$. Note that (\ref{item1:lem:tree_locally_compact}) is true if and only if for all $v \in \VV_1$ and $w \in \VV_2$ we have that all orbits of $G_w \big|_{B(w)}$ are finite and all suborbits of $G_v \big|_{B(v)}$ are finite. On the other hand (\ref{item2:lem:tree_locally_compact}) is true if and only if $G \Big|_{\VV_2}$ is subdegree-finite.
An argument similar to that used in the proof of Proposition~\ref{prop:finite_subdegrees} can therefore be used to show that
(\ref{item1:lem:tree_locally_compact}) and (\ref{item2:lem:tree_locally_compact}) are equivalent.
\end{proof}

Henceforth in this section, $T$ is the $(|\sta|,|\stb|)$-biregular tree, $\col$ is a legal colouring of $T$ with $\sta$ and $\stb$, and we bipartition the vertices of $T$ into sets $\Va$ and $\Vb$ in the usual way.

\begin{theorem} \label{thm:TopoProperties1}
Suppose $\gpa \leq \sym(\sta)$ and $\gpb \leq \sym(\stb)$ are
closed with nontrivial degree.
Then $\Boxproduct{\gpa}{\gpb}{}$ is closed in $\sym(\Vb)$ and the following are equivalent:
\begin{enumerate}
\item \label{Item:PtStabCompact}
	every point stabiliser in $\Boxproduct{\gpa}{\gpb}{}$ is compact in $\sym(\Vb)$;
\item
	for all $w \in \Vb$ the point stabiliser $(\Universal{\gpa}{\gpb}{})_w$ is compact in $\aut T$;
\item \label{Item:NCompact}
	$\gpb$ is compact and every point stabiliser in $\gpa$ is compact.
\end{enumerate}
\end{theorem}

\begin{proof} The result follows immediately from Lemma~\ref{lem:product_is_closed}, Lemma~\ref{lem:tree_locally_compact} and the fact that  $\Boxproduct{\gpa}{\gpb}{}$ and $\Universal{\gpa}{\gpb}{}$ are isomorphic as topological groups.
\end{proof}

\begin{theorem} \label{thm:TopoProperties2}
Suppose $\gpa \leq \sym(\sta)$ and $\gpb \leq \sym(\stb)$ are closed with nontrivial degree.
Then the following are equivalent:
\begin{enumerate}
\item \label{item:a:thm:TopoProperties}
	$\Universal{\gpa}{\gpb}{}$ is locally compact;
\item \label{item:b:thm:TopoProperties}
	all point stabilisers in $\gpa$ and all point stabilisers in $\gpb$ are compact; and
\item \label{item:c:thm:TopoProperties}
	the stabiliser in $\Universal{\gpa}{\gpb}{}$ of any edge in $\T$ is compact.
\end{enumerate}
\end{theorem}

\begin{proof}
It is obvious that (\ref{item:c:thm:TopoProperties}) implies (\ref{item:a:thm:TopoProperties}).
Suppose $G := \Universal{\gpa}{\gpb}{\col}$ is locally compact. Then there exists some finite set $\Phi$ of vertices such that $G_{(\Phi)}$ has all orbits finite. Fix $\pta \in \sta$ and recall the definition of the function $\pi_{\mathcal{P}}$ for
graph path
$\mathcal{P}$ in $T$ from Section~\ref{section:Simplicity}. Now $VT \setminus \Phi$ contains infinitely many vertices $v$ such that $\col \big|_{\overline{\ArcsFrom}(v)} = \pta$ and therefore there is an edge $\mathcal{P} := \{v,w\}$ in $T$ such that $\Phi$ is contained in $\pi^{-1}_{\mathcal{P}}(v)$ and $\col \big|_{\overline{\ArcsFrom}(v)} = \pta$. Note that $w \in \Va$. By Theorem~\ref{thm:prod_has_prop_P}, $G$ has Tits' independence property (P) and so the pointwise stabilisers $G_{(\pi^{-1}_{\mathcal{P}}(v))}$ and $G_{(v,w)}$ induce the same permutation group on $B(w)$, and this induced group is permutation isomorphic to $\gpa_\pta$. Now $G_{(\pi^{-1}_{\mathcal{P}}(v))} \leq G_{(\Phi)}$ and all orbits of $G_{(\Phi)}$ are finite, so it follows that all orbits of $\gpa_\pta$ are finite. For $\ptb \in \stb$, a symmetric argument shows that all orbits of $\gpb_\ptb$ are finite. Hence $\gpa_\pta$ and $\gpb_\ptb$ are compact
and we have shown that (\ref{item:a:thm:TopoProperties}) implies (\ref{item:b:thm:TopoProperties}).

Now suppose that all point stabilisers in $\gpa$ and $\gpb$ are compact; in particular this means that all orbits of point stabilisers in $\gpa$ and in $\gpb$ are finite. Choose two adjacent vertices $v,w \in VT$. We claim that $G_{(v,w)}$ has only orbits of finite length, from which it follows that
$G_{(v,w)}$ is compact.
Fix some vertex $u \in VT\setminus \{v,w\}$ and note that either $u$ is closest to $v$, or $u$ is closest to $w$. Without loss of generality, let us suppose that $u$ is closest to $w$; then the unique geodesic from $v$ to $u$ contains $w$. Let $y_0, y_1, y_2, \ldots, y_n$ denote this geodesic, with $y_0 = v, y_1 = w$ and $y_n = u$. Now
\[
|G_{(v,w)} u| = |G_{(y_0,y_1)} : G_{(y_0,y_1,y_2,\ldots,y_n)}| \leq \prod_{i=0}^{n-2} |G_{(y_i, y_{i+1})} y_{i+2}|,
\]
and each orbit $|G_{(y_i, y_{i+1})} y_{i+2}|$ in the product is finite because $y_i, y_{i+2} \in B(y_{i+1})$ and $G_{(y_i, y_{i+1})} \big|_{B(y_{i+1})}$ is permutation isomorphic to a point stabiliser in $\gpa$ or a point stabiliser in $\gpb$.
We have established our claim, and have thus shown that (\ref{item:b:thm:TopoProperties}) implies (\ref{item:c:thm:TopoProperties}).
\end{proof}

An example of a closed subgroup $\gpa \leq \sym(\sta)$ that is locally compact with non-compact point stabilisers is provided by any sharply $2$-transitive group of permutations of a countably infinite set.

\begin{theorem} \label{thm:tree_compactly_generated}
Suppose $\gpa \leq \sym(\sta)$ and $\gpb \leq \sym(\stb)$ have nontrivial degree and, under the permutation topology, both groups are closed with compact point stabilisers. Then all point stabilisers in $\Universal{\gpa}{\gpb}{}$ are compactly generated if and only if $\gpa$ and $\gpb$ are both compactly generated.

Moreover, if $\gpa$ and $\gpb$ are both compactly generated with only finitely many orbits, and $\gpa$ or $\gpb$ is transitive, then $\Universal{\gpa}{\gpb}{}$ is compactly generated.
\end{theorem}

\begin{proof} Write $G:=\Universal{\gpa}{\gpb}{\col}$.
Since $\gpa$ and $\gpb$ are closed with compact point stabilisers, $G$ is locally compact and all edge stabilisers in $G$ are compact by Theorem~\ref{thm:TopoProperties2}.

Let $v$ be a vertex in $\T$, and suppose that $G_v$ is compactly generated. Fix some vertex $w$ adjacent to $v$ in $\T$. If $S$ is a compact generating set for $G_v$, then $\bigcup_{s \in S} s G_{(v,w)}$ is an open cover of $S$, and therefore one can find a finite subset $\{s_1, \ldots s_r\}$ of $S$ such that $G_v = \langle G_{(v,w)}, s_1, \ldots, s_r \rangle$. If $v \in \Va$ then $G_v \big|_{B(v)}$ is permutationally isomorphic to $\gpa$ by Lemma~\ref{lem:product_locally_MN},
and under this isomorphism it is clear that $G_{v,w}$ is mapped to a point stabiliser in $\gpa$ (which is compact by assumption). Hence $\gpa$ is compactly generated. If $v \in \Vb$ then a symmetric argument shows that $\gpb$ is compactly generated.

Now assume instead that $\gpa$ and $\gpb$ are compactly generated. Suppose $v \in \Va$, choose $w \in B(v)$, and set $\pta := \col(v, w) \in \sta$. The stabiliser $\gpa_\pta$ is compact and open. If $S$ is a compact generating set for $\gpa$, then there exists some finite subset $\{s_1, \ldots, s_r\} \subseteq S$ such that
$\gpa = \langle \gpa_\pta, s_1, \ldots, s_r \rangle$.
By Lemma~\ref{lem:product_locally_MN},
$G_v \big|_{B(v)}$ is permutationally isomorphic to $\gpa$ 
and under this isomorphism $G_{(v,w)}\big|_{B(v)}$ is mapped to $\gpa_{\pta}$.
It follows that there exists $\{g_1, \ldots, g_r\} \subseteq G_v$ such that $H := \langle G_{(v,w)}, g_1, \ldots, g_r \rangle$ and $G_v$ induce the same permutation group on $B(v)$.
Hence, for any $g \in G_v$ there exists some $h \in H$ such that $w^g = w^h$, and so $gh^{-1} \in G_{(v,w)} \leq H$. Thus $G_v$ is equal to $H$, and therefore $G_v$ is compactly generated. A symmetric argument shows that if $v \in \Vb$, then $G_v$ is compactly generated.

Finally we note that if $\gpa$ and $\gpb$ have only finitely many orbits and either $\gpa$ or $\gpb$ (or both) is transitive, then
by Lemma~\ref{lem:QuotientIsAGraph},
the quotient $G \backslash \T$ is a finite tree, and so $G$ is generated by the vertex groups of $G \backslash \T$. It follows that there exists a finite set $\Phi \subseteq VT$ such that $G = \langle G_v : v \in \Phi \rangle$. Since $\Phi$ is finite and $G_v$ is compactly generated for all $v \in \Phi$, we have that $G$ is compactly generated.
\end{proof}

Recall that a permutation group is {\em semi-regular} if every point stabiliser is trivial.
A semi-regular action is also called a {\em free} action. 

\begin{theorem} \label{thm:discrete}
Suppose $\gpa \leq \sym(\sta)$ and $\gpb \leq \sym(\stb)$
have nontrivial degree.
Then the following are equivalent: 
(i) $\Universal{\gpa}{\gpb}{}$
is discrete; (ii) $\gpa$ and $\gpb$ are semi-regular; and (iii) $\Universal{\gpa}{\gpb}{}$ acts semi-regularly on the edges of $\T$.
\end{theorem}

\begin{proof} 
Let 
$G:=\Universal{\gpa}{\gpb}{\col}$. Suppose $\gpa$ and $\gpb$ are both semi-regular. If $w, w' \in \Vb$ are at distance two in $\T$, there is a unique element $v \in \Va$ that is adjacent to both vertices, and $G_{w, w'} = G_{w, v, w'}$. By Lemma~\ref{lem:product_locally_MN}, $G_{w, v}$ fixes $B(w)$ and $B(v)$ pointwise, and, since $T$ is connected, $G_{w, v}$ must therefore fix $VT$ pointwise; it is thus trivial. Hence
$\Universal{\gpa}{\gpb}{}$ is discrete.

Let $\Phi \subseteq \Vb$ be any finite set.
Suppose $\gpa$ is not semi-regular, and choose $\pta \in \sta$ such that $\gpa_\pta$ is nontrivial. 
For each $w' \in VT$, the connected components of $T \setminus B(w')$ each contain infinitely many vertices $w \in \Vb$ such that $\col \big|_{\overline{A}(w)} = \pta$, so we may choose a pair of adjacent vertices $v \in \Va, w \in \Vb$ such that $\col(v, w) = \pta$ and $\Phi$ is contained in the half-tree $T_{(w, v)}$.
Choose $\sigma \in \gpa_\pta \setminus \langle 1 \rangle$. By Lemma~\ref{lem:product_locally_MN}, there exists $h \in G_v$ such that $\col \big|_{A(v)} h \big|_{A(v)} \col \big|_{A(v)}^{-1} = \sigma$. Hence $h \in G_{v, w}$ and $h \not \in G_{(B(v))}$. By Lemma~\ref{lem:prod_has_independence}, there exists $g \in G_{v, w}$ such that $g \big|_{T_{(w, v)}}$ is trivial and $g \big|_{T_{(v, w)}} = h \big|_{T_{(v, w)}}$. Hence $g \in G_{(\Phi)}$ is nontrivial. A similar argument shows that if $\gpb$ is not semi-regular, then $G_{(\Phi)}$ is nontrivial. It follows that if $\gpa$ or $\gpb$ is not semi-regular, then
$\Universal{\gpa}{\gpb}{}$
is not discrete.
\end{proof}

\begin{corollary} Suppose $\sta$ and $\stb$ are countable, and $\gpa \leq \sym(\sta)$ and $\gpb \leq \sym(\stb)$ are closed. Then
\[|\Boxproduct{\gpa}{\gpb}{}| \leq \aleph_0 \quad \text{ or }  \quad |\Boxproduct{\gpa}{\gpb}{}| = 2^{\aleph_0},\]
with the former holding if and only if  $\gpa$ and $\gpb$ are semi-regular.
\end{corollary}

\begin{proof} By Theorem~\ref{thm:TopoProperties1}, the group $G := \Boxproduct{\gpa}{\gpb}{}$ is a closed group of permutations of the countably infinite set $\Vb$. Since the trivial group $\langle 1 \rangle$ is closed, we may deduce from a theorem of
David M.~Evans
(\cite[Theorem 1.1]{evans87}) that either $|G| = 2^{\aleph_0}$ or the pointwise stabiliser (in $G$) of some finite subset of $\Vb$ is trivial. 
But the pointwise stabiliser of some finite subset of $\Vb$ is trivial if and only if $G$ is discrete.
The corollary now follows from Theorem~\ref{thm:discrete}.
\end{proof}

%
%
\section{An uncountable set of compactly generated simple groups}

In \cite{CapraceDeMedts}, $\mathcal{S}$ is defined to be the set of
non-discrete, topologically simple, 
totally disconnected, locally compact groups.
Pierre-Emmanuel Caprace and Tom De Medts
remark that they do not know whether $\mathcal{S}$ contains uncountably many pairwise non-isomorphic compactly generated groups. In this section we prove that $\mathcal{S}$
contains uncountably many isomorphism classes
by constructing the first uncountable set of pairwise non-isomorphic simple groups that are non-discrete, compactly generated, totally disconnected and locally compact. In fact, we construct $2^{\aleph_0}$ such groups.

All the simple groups constructed in this paper have compact open subgroups that split nontrivially as direct products. The paper \cite{CapraceReidWillis} provides general results that apply to all compactly generated groups in $\mathcal{S}$ that enjoy the latter property; for example, they are all abstractly simple and non-amenable. Examples of compactly generated groups in $\mathcal{S}$ that are structurally similar to the groups we describe here are due to
Marc Burger and Shahar Mozes
(\cite{BurgerMozes}),
R\"{o}gnvaldur G.~M\"{o}ller and Jan Vonk
(\cite{mollervonk}), and
Christopher Banks, Murray Elder and George A.~Willis
(\cite{BanksElderWillis}).\\

Recall from \cite[pp. 58]{serre:trees} the definition of Serre's property (FA). Suppose $G$ is a group acting on a tree $T$ without inversion. Let $\Fix_G(T)$ be the set of vertices of $T$ that are fixed by all elements in $G$. A group $G$ has property (FA) if $\Fix_G(T)$ is non-empty for any tree $T$ on which $G$ acts without inversion. 
There is a topological version of (FA), which we denote here by (FA$_t$), whereby a topological group $G$ has property (FA$_t$) if every continuous action of $G$ on a tree has either a fixed vertex or a (setwise) fixed edge (see \cite[Definition 2.3.4]{propertyT}, for example). Compact topological groups obviously satisfy (FA$_t$), and it is also clear that a topological group $G$ with a closed normal subgroup $N$ such that $N$ and $G/N$ both have (FA$_t$), itself enjoys (FA$_t$).

A particularly rich class of discrete groups that all have property (FA) is the class of finitely generated torsion groups.

\begin{proposition}[{\cite[Example 6.3.1]{serre:trees}}] \label{serre_fg_torsion_gps_have_FA} A finitely generated torsion group has property (FA). \qed
\end{proposition}

\begin{theorem}[{\cite[Theorem 28.7]{olshanski}}] \label{thm:monster} For every sufficiently large prime number $p$, there is a continuum of pairwise non-isomorphic infinite groups of exponent $p$ all of whose proper nontrivial subgroups have order $p$. \qed
\end{theorem}

In \cite[pp. 304]{olshanski}
Ol'shanski{\u\i} 
notes that taking $p > 10^{75}$ in the above theorem is sufficient; let us now fix such a prime $p$. The groups whose existence is guaranteed by Theorem~\ref{thm:monster} we shall call
{\em Tarski--Ol'shanski{\u\i} Monsters}.

Let $Q$ be a 
Tarski--Ol'shanski{\u\i} Monster.
It is $2$-generated and torsion, so by Proposition~\ref{serre_fg_torsion_gps_have_FA} it has property (FA). It is easily seen to be simple, since if $N$ is a proper normal subgroup of $Q$, then $Q/N$ is infinite and torsion so it contains a nontrivial finite subgroup $H/N$, with $N < H < Q$.\\

Determining whether or not two box products are isomorphic can be difficult; it will be explored in detail in subsequent papers. For our purposes in this paper, we note two elementary results: Lemma~\ref{lem:subgroup_iso} and Lemma~\ref{lemma:topo_fa}.

\begin{lemma} \label{lem:subgroup_iso} Let $\gpa_1, \gpa_2, \gpb_1, \gpb_2$ be
permutation groups of nontrivial degree, with $\gpa_1$ nontrivial.
Suppose $\gpa_1$ has property (FA), and no nontrivial quotient of $\gpa_1$ is isomorphic to any subgroup of $\gpa_2$ or $\gpb_2$. Then
$\Universal{\gpa_1}{\gpb_1}{}$
and 
$\Universal{\gpa_2}{\gpb_2}{}$
are not (abstractly) isomorphic.
\end{lemma}

\begin{proof} For $i = 1, 2$, suppose $\gpa_i \leq \sym(\sta_i)$ and $\gpb_i \leq \sym(\stb_i)$ and let $T_i$ denote the $(|\sta_i|, |\stb_i|)$-biregular tree. Let $\col_i$ be a 
legal colouring with $\sta_i$ and $\stb_i$, 
and write $G := \Universal{\gpa_1}{\gpb_1}{\col_1}$ and $H := \Universal{\gpa_2}{\gpb_2}{\col_2}$.

Let us suppose, for a contradiction, that $G$ and $H$ are isomorphic. By Proposition~\ref{iso_subgroups}, $G$ contains a subgroup that is isomorphic to $\gpa_1$, and so $\gpa_1$ is isomorphic to some subgroup $K$ of $H$.

On one hand, $K$ cannot fix any vertex in $T_2$. Indeed, if $K$ fixes some vertex $v \in T_2$, then 
$K / K_{(B(v))} \cong K \big|_{B(v)}$
$\leq H_v \big|_{B(v)}$. Since $H_v \big|_{B(v)}$ is isomorphic to $\gpa_2$ or $\gpb_2$, this is only possible if the quotient $K / K_{(B(v))}$ is trivial; that is, if $K$ fixes $B(v)$ pointwise. But if this is so, then we can repeat this argument, since $K$ now fixes $w \in B(v)$. Since $T_2$ is connected, it follows that $K$ must be trivial, which is absurd.

On the other hand, $K$ must fix a vertex in $T_2$, since $K$ acts on $T_2$ without inversion and has property (FA).
\end{proof}

\begin{lemma} \label{lemma:topo_fa}  Let $\gpa_1, \gpa_2, \gpb_1, \gpb_2$ be closed permutation groups of nontrivial degree such that, under the permutation topology, $\gpb_i$ and all point stabilisers of $\gpa_i$ are compact for $i=1,2$.
Suppose each $\gpa_i$ is transitive, has the topological property (FA$_t$) and admits no nontrivial compact normal subgroups. If $\gpa_1$ and $\gpa_2$ are not topologically isomorphic, then $\Universal{\gpa_1}{\gpb_1}{}$ and $\Universal{\gpa_2}{\gpb_2}{}$ are not topologically isomorphic.
\end{lemma}

\begin{proof} We adopt the notation used in the first paragraph of the proof of Lemma~\ref{lem:subgroup_iso}. For $i = 1,2$, fix $v_i \in V_{\sta_i} \subseteq VT_i$,
and note that by Theorem~\ref{thm:TopoProperties2}, $G_{(B(v_1))}$ and $H_{(B(v_2))}$ are proper compact and normal subgroups of (respectively) the non-compact groups $G_{v_1}$ and $H_{v_2}$.

A compact normal subgroup $K$ of $G_{v_1}$ has only orbits of finite length; the closure of $K \big|_{B(v_1)}$ in $\sym(B(v_1))$ is thus compact and normal in $\gpa_1$, and so by assumption it is trivial. Hence every compact normal subgroup of $G_{v_1}$ is contained in $G_{(B(v_1))}$. A symmetric argument holds for $H_{v_2}$ and $H_{(B(v_2))}$.

Suppose $\varphi: G \rightarrow H$ is a topological isomorphism. The compact group $G_{(B(v_1))}$ has (FA$_t$) and $G_{v_1} / G_{(B(v_1))}$ (which is topologically isomorphic to $\gpa_1$ by Proposition~\ref{prop:quotient_iso_to_M}) also has (FA$_t$). By this, and a symmetric argument for $H$, we see that all point stabilisers in $G$ and $H$ have (FA$_t$).

Now $\varphi$ and $\varphi^{-1}$ are continuous actions of $G$ and $H$ respectively, so there exists $v_1^\ast \in VT_2$ and $v_1^{\ast\ast} \in VT_1$ such that $\varphi(G_{v_1}) \leq H_{v_1^\ast} \leq \varphi(G_{v_1^{\ast\ast}})$. Since $\gpa_1$ is transitive, $G_{v_1}$ fixes no vertex in $T_1\setminus \{v_1\}$ so we must have that $v_1 = v_1^{\ast\ast}$. Hence $\varphi(G_{v_1}) = H_{v_1^\ast}$. Since $G_{v_1}$ is not compact, Theorem~\ref{thm:TopoProperties1} implies that $v_1^\ast \in V_{\sta_2}$.
Let $v_2 := v_1^\ast \in V_{\sta_2}$ and note that $\varphi(G_{v_1}) = H_{v_2}$.

Now $\varphi(G_{(B(v_1))})$ is a compact normal subgroup of $H_{v_2}$ so it is contained in $H_{(B(v_2))}$, and $\varphi^{-1}(H_{(B(v_2))})$ is a compact normal subgroup of $G_{v_1}$ so it is contained in $G_{(B(v_1))}$. Hence
$\varphi(G_{(B(v_1))}) = H_{(B(v_2))}$. It follows then that $G_{v_1} / G_{(B(v_1))}$ and $H_{v_2} / H_{(B(v_2))}$ are isomorphic as topological groups. The result now follows immediately from Proposition~\ref{prop:quotient_iso_to_M}.
\end{proof}

Note that the converse to Lemma~\ref{lemma:topo_fa} is clearly false, as one should expect in light of Theorem~\ref{thm:discrete}. Indeed, a countably infinite simple group $G$ with a finite nontrivial subgroup has two different transitive permutation representations in $\sym(\mathbb{N})$, one with all stabilisers trivial which we denote by $\gpa_1$, and another with all stabilisers finite and nontrivial which we denote by $\gpa_2$. Both $\gpa_1$ and $\gpa_2$ are discrete under the permutation topology and are therefore isomorphic as topological groups. However, if $\gpb$ is a semi-regular permutation group of degree at least two, then the group $\Universal{\gpa_1}{\gpb}{}$ is discrete but $\Universal{\gpa_2}{\gpb}{}$ is not.\\

Let $Q$ be a Tarski--Ol'shanski{\u\i} Monster group and fix a nontrivial proper subgroup $H \leq Q$. The group $Q$ is simple and it acts faithfully and transitively on the coset space $\sta := (Q : H)$. We call the permutation group induced by $Q$ on $\sta$ a {\em Tarski--Ol'shanski{\u\i} Monster permutation group}. By Theorem~\ref{thm:monster}, there are $2^{\aleph_0}$ pairwise (abstractly) non-isomorphic Tarski--Ol'shanski{\u\i} Monster permutation groups.

For the following theorem, note that the trivial subgroup of a permutation group of nontrivial degree is generated by point stabilisers, as is any non-regular primitive permutation group.

\begin{theorem} \label{thm:infinitely_many_simple_groups} 
There are $2^{\aleph_0}$ (abstractly) non-isomorphic groups of the form $\Universal{Q}{\gpb}{}$, where $Q$ is a Tarski--Ol'shanski{\u\i} Monster permutation group and 
$\gpb$ is a finite permutation group of nontrivial degree that is generated by point stabilisers. Furthermore, each such group $\Universal{Q}{\gpb}{}$ under its permutation topology is totally disconnected, locally compact, compactly generated, simple and not discrete.
\end{theorem}

\begin{proof}
Let $Q \leq \sym(\sta)$ be a 
Tarski--Ol'shanski{\u\i} Monster permutation group
and let $\gpb \leq \sym(\stb)$ be a finite permutation group of nontrivial degree $d$ that is generated by point stabilisers.
We think of $Q$ as a topological group under its permutation topology. Note that $Q$ is transitive.
Point stabilisers in $Q$ are finite, so $Q$ is totally disconnected and locally compact, with compact stabilisers. Moreover, $Q$ is finitely generated, so it is compactly generated. 

Let $T$ be the
$(\aleph_0, d)$-biregular tree,
and bipartition the vertices of $T$ into sets $\Va$ and $\Vb$ in the usual way. The group
$\Universal{Q}{\gpb}{}$
is a subgroup of $\aut T$.
Under the permutation topology
it inherits from $\aut T$, the group $\Universal{Q}{\gpb}{}$
is totally disconnected because its action is faithful. It is locally compact by Theorem~\ref{thm:TopoProperties2}, it is compactly generated by Theorem~\ref{thm:tree_compactly_generated}, and it is non-discrete by Theorem~\ref{thm:discrete}.

Any point stabiliser in $Q$ is a maximal subgroup, so $Q$ is generated by point stabilisers; so too is
$\gpb$.
Hence, by Theorem~\ref{theorem:simplicity}, the group
$\Universal{Q}{\gpb}{}$
is simple.

If $Q'$ is another
Tarski--Ol'shanski{\u\i} Monster 
that is not isomorphic to $Q$, then
$\Universal{Q}{\gpb}{}$
and
$\Universal{Q'}{\gpb}{}$
are not isomorphic by Lemma~\ref{lem:subgroup_iso}. The theorem now follows immediately from Theorem~\ref{thm:monster}.
\end{proof}

Recall that $\mathcal{S}$ denotes the class of non-discrete, totally disconnected, locally compact, topologically simple groups. Because all compactly generated groups in $\mathcal{S}$ are Polish groups, and it is well-known that there are precisely $2^{\aleph_0}$ 
many topological isomorphism classes of Polish groups, we have the following corollary. 

\begin{corollary} There are precisely $2^{\aleph_0}$ topological group isomorphism types of compactly generated groups in $\mathcal{S}$.
\qed
\end{corollary}

\begin{remark} \label{remark:construction} Of course, the construction described above works for many groups that are not
Tarski--Ol'shanski{\u\i} Monsters.
The following is a general method for constructing examples of non-discrete
abstractly simple
groups that are totally disconnected, locally compact, and compactly generated.
\begin{enumerate}
\item
Choose $\gpa \leq \sym(\sta)$ to be
closed
transitive and non-regular,
with all point stabilisers compact,
such that $\gpa$ is compactly generated and generated by point stabilisers. Note that we do not require  $M$ to be non-discrete.

Example: Take $M$ be any compactly generated totally disconnected locally compact group and $U$ be a nontrivial compact open subgroup of $M$, properly contained in $M$, such that $M$ is generated by the union of the conjugacy class of $U$ and $U$ has trivial core in $M$. One may then set $X$ to be the coset space $(M:U)$ and view $M$ as a closed subgroup of $\rm{Sym}\,(X)$ via the action by left multiplication. 

Example: Take $\gpa \leq \sym(\sta)$ to be a closed subdegree-finite non-regular primitive permutation group of degree at least three.

\item
	Choose any 
	finite permutation
	group $\gpb \leq \sym(\stb)$ of degree at least two, that is generated by point stabilisers. For example, one could take $\gpb$ be be any finite primitive non-regular permutation group of degree at least three. In the permutation topology, such a group will be totally disconnected and compact.
\item
	Choose any legal colouring $\col$ of the $(|\sta|, |\stb|)$-biregular tree $T$, and let $\Vb$ be
a part of the bipartition of $T$ whose elements have valence $|\stb|$.
\item
	The group
	$\Universal{\gpa}{\gpb}{} \leq \aut T$
	is simple (by Theorem~\ref{theorem:simplicity}), and with the permutation topology it is totally disconnected (because it is faithful), locally compact (by Theorem~\ref{thm:TopoProperties1}), compactly generated (by Theorem~\ref{thm:tree_compactly_generated}) and non-discrete (by Theorem~\ref{thm:discrete}). Point stabilisers
of vertices in $\Vb$
are compact.
\end{enumerate}
\end{remark}

\begin{remark} 
In \cite[Problem 4.3]{willis2007},
George A.~Willis
asks the following. Let $G_1$, $G_2$ be (non-discrete) topologically simple (or simple) totally disconnected locally compact groups. Suppose that there are compact open subgroups $U_i \leq G_i$ ($i = 1, 2$) that are isomorphic. Does it follow that $G_1$ and $G_2$ are isomorphic? It is known (see \cite{BEW} and \cite{CapraceDeMedts}) that the answer to this question is no. However, only countably many pairs $(G_1, G_2)$ demonstrating this have been found. Using the box product, it is easy to construct an uncountable set of such groups $\{G_i : i \in I\}$, in which each group $G_i$ contains a compact open subgroup $U_i \leq G_i$ such that for all $i, j \in I$ the groups $U_i$ and $U_j$ are isomorphic as topological groups but $G_i$ and $G_j$ are non-isomorphic.

Indeed, for two groups
$G:=\Universal{Q}{S_3}{}$ and $H:=\Universal{Q'}{S_3}{}$
taken from the proof of Theorem~\ref{thm:infinitely_many_simple_groups}, any point stabiliser
in $G$ of a vertex $v \in \Vb$
is permutation isomorphic to any point stabiliser in $H$. This is easy to see if you consider the induced actions of G and H on the $(\aleph_0, 3)$-biregular tree $T$.

Thus, if $Q$ and $Q'$ are non-isomorphic, then $G$ and $H$ are non-isomorphic, non-discrete, simple, totally disconnected locally compact groups and moreover $G$ and $H$ have compact open subgroups, $G_v$ and $H_v$ respectively, which are isomorphic as topological groups.
\end{remark} 

\\

{\bf Acknowledgments} The author would like to thank R\"{o}gnvaldur G.~M\"{o}ller
for providing him with an English translation of Tits' paper (\cite{tits}).
The author would also like to thank the anonymous referees for their thoughtful comments and suggestions.

\end{document}